\documentclass[11pt, oneside, reqno]{amsart}
\usepackage{amsmath,amsthm,amssymb,amsfonts, amsaddr}
\usepackage[shortlabels, inline]{enumitem}
\usepackage{mathtools, mathrsfs, xargs}
\usepackage[dvipsnames]{xcolor}
\usepackage{hyperref}
\hypersetup{colorlinks, citecolor=blue, linkcolor=blue, urlcolor=blue}
\newif\iflabels
\iflabels
  \usepackage[left=0.5in, right=2.5in, marginparwidth=2.2in]{geometry}
  \usepackage{showlabels}
\else
  \usepackage[left=1.10in, right=1.10in, bottom=1in, top=1in]{geometry}
\fi

\numberwithin{equation}{section}
\theoremstyle{plain}                
\newtheorem{theorem}{Theorem}[section]
\newtheorem{lemma}[theorem]{Lemma}
\newtheorem{proposition}[theorem]{Proposition}

\newtheorem{assumption}[theorem]{Assumption}

\theoremstyle{definition}

\newtheorem{question}[theorem]{Question}

\theoremstyle{remark}
\newtheorem{remark}[theorem]{Remark}

\usepackage{xargs} 
\usepackage{mathtools}

\selectfont

\providecommand{\alias}{}
\renewcommand{\alias}[1]{\providecommand{#1}{}\renewcommand{#1} }

\DeclarePairedDelimiter\ab{\langle}{\rangle} 
\DeclarePairedDelimiter\abs{\lvert}{\rvert}   
\DeclarePairedDelimiter\norm{\lVert}{\rVert}  
\DeclarePairedDelimiter\bkt{[}{]}             
\DeclarePairedDelimiter\brc{\{}{\}}           
\DeclarePairedDelimiter\prn{(}{)}             

\providecommand\giv{}

\alias\given\giv 

\DeclarePairedDelimiterXPP\pp[1]{\mathbb{P}}[]{}{
   \renewcommand\giv{\nonscript\:\delimsize\vert\nonscript\:\mathopen{}}
   #1}
\DeclarePairedDelimiterXPP\ppup[2]{\mathbb{P}^{#1}}[]{}{
   \renewcommand\giv{\nonscript\:\delimsize\vert\nonscript\:\mathopen{}}
   #2}
\DeclarePairedDelimiterXPP\ppdown[2]{\mathbb{P}_{#1}}[]{}{
   \renewcommand\giv{\nonscript\:\delimsize\vert\nonscript\:\mathopen{}}
   #2}
   \DeclarePairedDelimiterXPP\ppupdown[3]{\mathbb{P}^{#1}_{#2}}[]{}{
   \renewcommand\giv{\nonscript\:\delimsize\vert\nonscript\:\mathopen{}}
   #3}
   
\DeclarePairedDelimiterXPP\ee[1]{\mathbb{E}}[]{}{
   \renewcommand\giv{\nonscript\:\delimsize\vert\nonscript\:\mathopen{}}
   #1}

\DeclarePairedDelimiterXPP\eeup[2]{\mathbb{E}^{#1}}[]{}{
   \renewcommand\giv{\nonscript\:\delimsize\vert\nonscript\:\mathopen{}}
   #2}

\DeclarePairedDelimiterXPP\eedown[2]{\mathbb{E}_{#1}}[]{}{
   \renewcommand\giv{\nonscript\:\delimsize\vert\nonscript\:\mathopen{}}
   #2}

   \DeclarePairedDelimiterXPP\eeupdown[3]{\mathbb{E}^{#1}_{#2}}[]{}{
   \renewcommand\giv{\nonscript\:\delimsize\vert\nonscript\:\mathopen{}}
   #3}
\DeclarePairedDelimiterXPP\eeud[3]{\mathbb{E}^{#1}_{#2}}[]{}{
   \renewcommand\giv{\nonscript\:\delimsize\vert\nonscript\:\mathopen{}}
   #3}

\let\preexp\exp
\let\exp\relax

\DeclarePairedDelimiterXPP\exp[1]{\preexp}(){}{#1}

\newcommand{\Exp}[1]{e^{#1}}

\newcommandx{\prf}[2][2=t\in\bkt{0,T}]{ \set{#1}_{#2}}

\newcommandx{\seq}[2][2=n\in\N]{\brc*{#1}_{#2}}
\newcommandx{\sseq}[2][2=n\in\N]{\{#1\}_{#2}}

\newcommandx{\seqk}[1]{\seq{#1}[k\in\N]}
\newcommandx{\seqkz}[1]{\seq{#1}[k\in\N_0]}

\newcommandx{\sseqk}[1]{\sseq{#1}[k\in\N]}
\newcommandx{\sseqkz}[1]{\sseq{#1}[k\in\N_0]}

\newcommandx{\fml}[2]{\set{#1}_{#2}}
\newcommandx{\tol}[1]{\stackrel{#1}{\longrightarrow}}
\newcommandx{\eqd}[1][1=d]{\stackrel{(#1)}{=}}
\newcommandx{\neqd}[1][1=d]{\stackrel{(#1)}{\neq}}

\makeatletter
\let\oldabs\abs \def\abs{\@ifstar{\oldabs}{\oldabs*}}
\let\oldab\ab \def\ab{\@ifstar{\oldab}{\oldab*}}
\let\oldnorm\norm \def\norm{\@ifstar{\oldnorm}{\oldnorm*}}
\let\oldbkt\bkt \def\bkt{\@ifstar{\oldbkt}{\oldbkt*}}
\let\oldbrc\brc \def\brc{\@ifstar{\oldbrc}{\oldbrc*}}
\let\oldprn\prn \def\prn{\@ifstar{\oldprn}{\oldprn*}}
\let\oldpp\pp \def\pp{\@ifstar{\oldpp}{\oldpp*}}
\let\oldppup\ppup \def\ppup{\@ifstar{\oldppup}{\oldppup*}}
\let\oldppdown\ppdown \def\ppdown{\@ifstar{\oldppdown}{\oldppdown*}}
\let\oldee\ee \def\ee{\@ifstar{\oldee}{\oldee*}}
\let\oldeeup\eeup \def\eeup{\@ifstar{\oldeeup}{\oldeeup*}}
\let\oldeedown\eedown \def\eedown{\@ifstar{\oldeedown}{\oldeedown*}}
\let\oldeeupdown\eeupdown \def\eeupdown{\@ifstar{\oldeeupdown}{\oldeeupdown*}}
\let\oldeeud\eeud \def\eeud{\@ifstar{\oldeeud}{\oldeeud*}}

\let\oldexp\exp \def\exp{\@ifstar{\oldexp}{\oldexp*}}
\makeatother

\newcommand{\Bprn}[1]{ \prn*[\Big]{#1} }
\newcommand{\bprn}[1]{ \prn*[\big]{#1} }

\newcommand{\set}[1]{ \brc{#1} }
\newcommand{\sets}[2]{ \brc{#1\,:\,#2} }

\alias{\R}{{\mathbb R}}
\alias{\C}{{\mathbb C}}
\alias{\Z}{{\mathbb Z}}
\alias{\N}{{\mathbb N}}
\alias{\Nz}{{\mathbb N}_0}

\DeclareMathOperator\Int{Int}

\DeclareMathOperator\Var{Var}

\DeclareMathOperator\Cov{Cov}
\DeclareMathOperator\corr{corr}

\newcommand{\oo}[1]{\frac{1}{#1}}
 
\newcommand{\too}[1]{\tfrac{1}{#1}}
\newcommand{\tot}{\too{2}} 

\newcommand{\inds}[1]{ 1_{\set{#1}}} 

\newcommand{\downto}{\searrow}
\newcommand{\upto}{\nearrow}

\newcommand{\al}{\alpha}
\newcommand{\be}{\beta}
\newcommand{\dl}{\delta}
\newcommand{\eps}{\varepsilon}
\newcommand{\ld}{\lambda}

\newcommand{\gm}{\gamma}
\newcommand{\vp}{\varphi}

\newcommand{\sD}{\mathcal{D}} 
 
\newcommand{\sF}{\mathcal{F}} 
\newcommand{\sG}{\mathcal{G}}

 \newcommand{\bN}{\mathbb{N}}
 
 \newcommand{\bP}{\mathbb{P}}
 \newcommand{\bQ}{\mathbb{Q}}

\newcommand{\sT}{\mathcal{T}}

\newcommand{\PP}{\bP}

\newcommandx{\pds}[2][1= ]{\frac{\partial #1}{\partial #2}}
\newcommandx{\tpds}[2][1= ]{\tfrac{\partial #1}{\partial #2}}

\newcommandx{\pdt}[2][1= ]{\frac{\partial^2 #1}{\partial #2}}
\newcommandx{\tpdt}[2][1= ]{\tfrac{\partial^2 #1}{\partial #2}}

\newcommandx{\pdm}[3][2= ]{\frac{\partial^{#1} #2}{\partial #3}}
\newcommandx{\tpdm}[4][2= ]{\tfrac{\partial^{#1} #2}{\partial #3}}

\newcommandx{\upp}[2]{{#1}^{(#2)}}

\newcommand{\efor}{\text{ for }}
\newcommand{\eforall}{\text{ for all }}
\newcommand{\eforeach}{\text{ for each }}
\newcommand{\eand}{\text{ and }}

\newcommand{\ewhere}{\text{ where }}
\newcommand{\ewith}{\text{ with }}

\newcommand{\cd}{c\` adl\` ag}

\newcommand{\dlink}[2]{\href{http://dlmf.nist.gov/#1}{#2}}
\newcommand{\cdg}[2]{\cite[\dlink{#1}{#2}]{DLMF}}

\newcommand{\cdl}[2]{\cdg{#1.E#2}{(#1.#2)}}

\newcommand{\eenup}[2]{\eeupdown{#1}{n}{#2}}
\newcommand{\eenmu}[1]{\eenup{m_n}{#1}}
\newcommand{\eennu}[1]{\eenup{\nu}{#1}}
\newcommand{\eennn}[1]{\eenup{\nu_n}{#1}}
\newcommand{\eenx}[1]{\eenup{x}{#1}}

\newcommand{\een}[1]{\eenup{}{#1}}

\newcommand{\eenmn}[1]{\eeupdown{m_n}{n}{#1}}

\newcommand{\ppnup}[2]{\ppupdown{#1}{n}{#2}}
\newcommand{\ppnx}[1]{\ppnup{x}{#1}}
\newcommand{\ppnnu}[1]{\ppnup{\nu_n}{#1}}

\newcommand{\ppn}[1]{\ppnup{}{#1}}

\newcommand{\jbn}{j_{\be-1,n}}
\newcommand{\jnn}{j_{\nu,n}}
\newcommand{\old}{\overline{\ld}}
\newcommand{\prfi}[1]{\{ #1 \}_{t \in [0,\infty)}}
\newcommand{\psim}{\psi^{\mu}} 
\newcommand{\zetam}{\zeta^{\mu}} 
\newcommand{\psiz}{\psi^0}
\newcommand{\seqn}[1]{\{#1\}_{n}}

\newcommand{\tgt}{T^{x_n,t}}
\newcommand{\tg}{\tilde{g}}

\newcommand{\umnpp}{(\umn)''}
\newcommand{\umnp}{(\umn)'}
\newcommand{\umn}{u^{\mu}_n}
\newcommand{\um}{u^{\mu}}
\newcommand{\vpm}{\vp^{\mu}} 
\newcommand{\vpz}{\vp^0}
\newcommand{\tm}{T^{\Delta}}
\newcommand{\wm}{w^{\mu}}

\newcommand{\movealign}[1]{\hspace{-#1em}&\hspace{#1em}}

\begin{document}

\title{A functional limit theorem for additive functionals}

\author{Thibaud Taillefumier}
\address{
  Department of Neuroscience and Center for Theoretical and
  Computational Neuroscience, and
  Department of Mathematics, The University of Texas at Austin}
\email{ttaillef@austin.utexas.edu}

\author{Gordan {\v Z}itkovi{\' c}}
\address{Department of Mathematics, The University of Texas at Austin}
\email{gordanz@math.utexas.edu}


\begin{abstract} We study a general limiting framework for the
  convergence of sequences of additive functionals of diffusions to
  L\'evy subordinators, and we provide explicit sufficient conditions
  that both ensure convergence and characterize the law of the limit.
  As an application, we identify a novel limiting regime for
  Wright–Fisher and Feller diffusions in the reflecting case and
  describe the corresponding limiting subordinator. This work is
  motivated by, and has applications in, neuroscience, where reflected
  diffusions are used to parametrize synchrony in doubly-stochastic
  models of spiking activity. 
\end{abstract}

\maketitle

\section{Introduction}
\subsection{Related Literature}
This work focuses on a class of limiting results for additive functionals
of recurrent diffusions and more general Markov processes. Problems of this type have been studied since at least the early 1950s, with particular emphasis on small- and large-parameter asymptotics for additive functionals of fixed diffusion processes. Some of the earliest contributions (see, e.g., \cite{KalRob53}) addressed the convergence of marginal distributions of integral functionals of Brownian motion. Subsequent developments shifted the focus to functional convergence and to more general underlying processes.

A foundational contribution in this direction is the work of
Papanicolaou, Stroock, and Varadhan \cite{PapStrVar77}, which
introduced the powerful martingale method. Kasahara and Kotani
\cite{KasKot79} further advanced the theory by considering
nonstandard normalizations and potentially nongaussian limits, and
by introducing $M_1$-con\-ver\-gen\-ce,   a mode of convergence
particularly suited for convergence to discontinuous limits, and one that we also adopt in the present work. These results were later extended by Yamada \cite{Yam86} to encompass a broader class of less regular integral functionals.

Around the same time, Kipnis and Varadhan \cite{KipVar86} laid the groundwork for a general framework for functional central limit theorems (FCLTs) and Gaussian limits in the setting of reversible diffusions. More recently, Cattiaux, Chafaï, and Guillin \cite{CatChaGui21} have provided a modern perspective on this framework and its subsequent developments, employing a PDE-based approach that complements the original probabilistic techniques. We refer the reader to their work for an extensive list of further references tracing the evolution of the Kipnis–Varadhan theory.
Although it is not strictly within the Markovian context, we also
cite the work of Hu, Nualart, and Xu \cite{HuNuaXu14} which extends
the martingale method of \cite{PapStrVar77} to integral functionals of
fractional Brownian motion.

As far as L\' evy-process limits are concerned, the literature seems
to be much less developed. A recent study most directly related to
ours is that of \cite{Bet23}, which establishes an FCLT for sequences
of integral functionals of one-dimensional diffusions, with stable
processes appearing as the limiting objects. That approach builds on
the method of \cite{FouTar18}, which leverages a representation of
one-dimensional diffusions as deterministic transformations of a
time-changed Brownian motion. This representation effectively reduces
the analysis to Brownian functionals and enables an elegant
probabilistic derivation of stable limits without any PDE input. We
also mention the work \cite{JarKomOll09}  of Jara, Komorovski and Olla
in the Markov-chain setting and \cite{MelMisMou11} in the setting of
kinetic theory. 

\subsection{Our contributions}

While situated within the broader trajectory of the literature
reviewed above, the present work adopts a distinct framework. Beyond
extending convergence results to encompass a wider class of additive
functionals and underlying diffusions, our approach diverges from the
standard FCLT-centered paradigm. Specifically, we allow both the
additive functional and the characteristics of the diffusion to vary
along the limiting sequence. This enables us to consider limiting
regimes that go beyond the classical center-and-rescale procedure
applied to a fixed process.

In this sense, our setting parallels the triangular array framework of
Kolmogorov and Gnedenko, which is capable of producing general
infinitely divisible limits, in contrast to the fixed-i.i.d.~scaling
that yields only stable laws. Analogously, our framework accommodates
a broader spectrum of limiting behavior by varying the input processes
themselves. To maintain sufficient generality while preserving
tractability, we restrict our attention to positive additive
functionals, which naturally lead to subordinator limits.

We now outline the two principal contributions of the paper: a general
convergence result for positive additive functionals of diffusions,
and an explicit convergence result for a specific sequence of
Wright–Fisher as well as Feller (CIR) diffusions.

\smallskip

\subsubsection{A general convergence result} 
We consider a sequence of one-dimensional
diffusions, each defined on some  interval in $\R$, as well as 
a sequence of associated 
nondecreasing continuous additive functionals. 
We impose no assumptions
on these diffusions or their boundary behavior except for positive
recurrence. Our main result (Theorem
\ref{main-2}) gives sufficient conditions on the characteristics of
the diffusions (speed measures and scale functions) and the
characteristics of the additive functionals (representing, i.e.,
Revuz 
measures) for convergence of the additive functionals to a L\' evy
subordinator. These conditions are stated in terms of the
limiting properties of the sequence of fundamental solutions associated with the diffusions
killed at ``rates'' dictated by the additive functionals. 
The killing operation allows us to combine each diffusion in the
sequence with its 
additive functional into a single killed diffusion which
can then be further analyzed. As a result, we derive a purely analytic
criterion for convergence in this framework, and give an expression
for the Laplace exponent of the limiting subordinator.

In addition to the idea of merging diffusions and their additive
functionals described above, the proof of Theorem \ref{main-2} relies
on an abstract convergence result in the $M_1$-topology (Theorem \ref{main-1}) for general
recurrent strong Markov processes. Beyond its role in later
developments, the strength of this, abstract, result lies in the fact
that it rests on a fairly elementary argument 
and
applies in great generality. It is stated in terms of resolvents of 
the additive functionals and, without relying on the existence (or
even the notion of) the local time, formalizes the following
intuition from the diffusion case: 
each additive functional in the sequence, time
changed by the inverse local time of its underlying diffusion (at an
appropriate ``reset point'') is a L\' evy subordinator. If laws of
these subordinators converge and the jumps of the inverse local times
shrink, the additive functionals themselves should converge to a
L\' evy subordinator.

Theorem \ref{main-1} demonstrates that the above program can be formalized
under the appropriate conditions on the initial distributions of
the diffusions and minimal tightness-type conditions on the
additive functionals. Moreover, 
it singles out Skorokhod's $M_1$-topology
 --- which has already appeared in a related context, (see, e.g.,
\cite{KasKot79}) ---
as the appropriate one in our setting. We note that one cannot hope for a
significantly stronger convergence,  such as $J_1$-convergence, 
than the one induced by the $M_1$-topology. Indeed, a typical
application of the theorem involves convergence of a sequence of
continuous processes to a discontinuous one. 

\subsubsection{The Wright-Fisher diffusion} 
The second part of the paper
introduces a novel limiting regime 
for a sequence of Wright-Fisher diffusions. The law of
each diffusion is determined by three parameters, $\al_n$, $\be_n$
and $\tau_n$;
$\al_n$ and $\be_n$ dictate
the shape of the stationary beta distribution, while $\tau_n$ plays a
role in scaling and time correlation. 
In our regime $\be_n = \be$ is fixed, while $\al_n$ and $\tau_n$
converge to $0$ at the same rate. The additive functionals we
study are simply the integrals $A_n(t) = \int_0^t X_n(u)\, du$. 

Relying on Theorem \ref{main-2} described above, we show that the
limiting subordinator exists in this regime and that its Laplace
functional can be expressed in terms of a quotient of modified Bessel
functions. We stress that our analysis is made
significantly more difficult by the fact that explicit expressions for
the fundamental solutions are not available for the killed
Wright-Fisher diffusion. Our approach takes advantage of the
polynomial nature of the Wright-Fisher diffusions and analyzes the
associated Poisson equations using series expansions and the recurrence
relations satisfied by their coefficients. 

To the best of our knowledge the limiting subordinator we obtained has
not appeared in the literature before. Interestingly, its Laplace
exponent appeared in a related but different context in
\cite[eq.~(48), p.~12]{PitYor03} where it is shown to be related to a
certain conditional distribution of a time-changed occupation time of
a Bessel process. The second part of the section
provides a detailed study of various properties of this subordinator:
we determine the range of finite moments of the jump measure, give an
explicit expression for its jump density in terms of the positive
zeros of the Bessel function, provide a computationally efficient
recursive formula for its moments, express its cumulants in terms of
the Rayleigh function, and exhibit an unexpectedly simple
continued-fraction expansion of the Laplace exponent.

In addition to our main example, the Wright-Fisher diffusion, we treat
a sequence of Feller diffusions and associated integral
functionals in a similar limiting parameter regime. Unlike in the 
Wright-Fisher case, fundamental solutions of killed Feller
diffusions admit explicit representations in terms of Kummer
functions. This makes it somewhat easier to prove convergence and
identify the limiting process in the family of inverse-Gaussian
subordinators. 

\subsection{Neuroscientific motivation}
The primary motivation for this work stems from recent attempts to quantitatively 
model correlated neural activity in neuroscience \cite{Becker,Becker2}. 
In these forays, configurations of $K$ neurons are modeled as random vectors 
$(B_{1}, \ldots,B_{K}) \in \{0,1\}^K$ with probability law
\begin{align}\label{eq:indMod}
    \mathbb{P}\left[B_{1}=b_1, \ldots, B_{K}=b_K \right]   =  
    \mathbb{E}\left[  \prod_{k=1}^K Z ^{b_{k}} (1-Z)^{1-b_{k}}  \right]
     ,
\end{align}
where $Z$ has the  distribution $F(dz)$ (called the mixing measure) 
supported by $[0,1]$.
The more dispersed the distribution $F$, the more correlated the spiking activity, 
a phenomenon that  can be quantified by remarking that $\Cov[B_k, B_l] = \Var[Z]$, $k\ne l$, 
so that the pairwise spiking correlation satisfies 
\begin{align*}
    \rho &=
    \corr\left[ B_{k}, B_{l} \right] = 
\mathrm{Var}\left[ Z \right] / (\mathbb{E}\left[Z\right]
(1-\mathbb{E}\left[Z\right] )) \efor k\ne l.
\end{align*}
By exchangeability of the variables $B_k$, correlations are entirely
encoded by the fluctuations of the total number of spiking neurons $S= \sum_{k=1}^K B_{k}$.
In turn, correlated spiking dynamics can be simply obtained by considering  
the sequence $\{S_j\}_{j\in\N}$, where  $\{S_j\}_{j\in\N}$ are iid copies of $S$.
That said, biophysically realistic models require one to consider continuous-time extensions 
of the above discrete-time dynamics, which are typically obtained as scaling limits.
In the iid setting, such a scaling limit is naturally specified by constructing the family $\{S^{\eps}\}_{j\in\N}$, $\eps>0$,
for a family of mixing distributions $\fml{F^{\eps}}{\eps>0}$ on $[0,1]$
 whose means scale linearly with $\eps$ as $\eps \downto 0$. 
The resulting scaling limits
\begin{align}\label{eq:compProc}
    Y(t) = \lim_{\eps \downto 0} \sum_{j=1}^{\lfloor t/\eps \rfloor}
    S^{\eps}_j, t\geq 0, 
\end{align} 
are  compound Poisson processes whose 
jumps  come at rate 
$\lim_{\eps \downto 0} (1-\pp{S^{\eps}=0})/\eps$, with the size $J$ of
each jump distributed as
$\PP[J=k]=\lim_{\eps \downto 0} \pp{ S^{\eps} = k   \giv S^{\eps}>0}$ for
$k=1,\dots, K$. The limiting spiking correlation can be backed out of this
distribution as follows:
\begin{align*}
    \lim_{\eps \downto 0} \rho^{\eps} = \frac{\ee{J(J-1)}}{(K-1)\ee{J}}\, .
\end{align*}

Although practically useful, the scaling limits presented above represent merely a special case
and their construction hinges on unrealistic iid simplifying assumptions.
The results presented in this manuscript address these limitations by constructing 
scaling limits for doubly-stochastic models of spiking activity, which are more realistic 
and do not assume an iid character in discrete time.
Indeed, these models consider total spiking counts defined as random variables 
$\{S_j\}_{j\in\N}$ with 
\begin{align}\label{eq:depMod}
    \mathbb{P}\left[S_1=s_1, \ldots, S_J=s_J \right]   =  
    \mathbb{E}{ \left[ \prod_{j=1}^J  \binom{K}{s_j} Z_j^{s_j} (1-
    Z_j)^{1-s_j} \right] },
\end{align}
for  $J\in\N$ and $s_1,\dots, s_J \in \set{1,\dots, K}$, 
where $Z_j=\int_{j-1}^j X_t \, dt$, $j \in \N$ and $\{X_t\}_{t\geq 0}$ is a continuous-time process with values in $[0,1]$.
In this approach, the process $Z$ represents the fluctuating, shared spiking rate of a neuronal population,
 typically modeled as a Wright-Fisher diffusion. 

\subsection{The structure of the paper.}
After this introduction, we develop a general convergence theorem for
reflected diffusions in Section \ref{sec:diffusions}, while a detailed
treatment of the sequence of Wright-Fisher and Feller diffusions is
left for Section \ref{sec:Wright-Fisher}.
Appendix \ref{app:proof} contains a proof of an abstract convergence
theorem for strong Markov processes.

\section{Sufficient conditions for convergence}
\label{sec:diffusions}

In this section we derive sufficient conditions on a sequence of
Markov processes and associated additive functionals for convergence
to a L\' evy subordinator. While our main focus is on the diffusion
framework later on in Theorem \ref{main-2}, we start the section with
a more general result (Theorem \ref{main-1}) for strong Markov processes.

\subsection{General convergence to a L\' evy subordinator}
\label{sse} For a metric space $E$, let $D(E)$ be the set of all \cd{}
functions $\omega:[0,\infty) \to E$, i.e., right-continuous functions
that admit left limits at all $t>0$. $D(E)$ comes naturally equipped
with the $\sigma$-algebra $\sD(E)$ generated by the evaluation maps
$X(t):D(E)\to E$, $X(t)(\omega) = \omega(t)$, as well as with the
family $\prfi{\theta(t)}$, of shift operators $\theta(t) : D(E) \to
D(E)$ given by $\prn{\theta(t)(\omega)}(u)= \omega(t+u)$ for
$t,u\geq 0$.

\smallskip

For $n\in \N$, let $E_n$ be a metric space, $x_n$ a point in $E_n$,
$\bP_n$ a probability measure on $D(E_n)$ and $\sF_n(t)$ a filtration
which contains the 
$\bP_n$-completion of the natural filtration of the canonical process
$X_n=\prfi{X_n(t)}$ made up of evaluation maps on $D(E_n)$. We assume
that $X_n$ is a time-homogeneous strong $\sF_n$-Markov process under $\bP_n$
for each $n\in \N$. More precisely, we assume that for each bounded
random variable $G$ on $D(E_n)$, there exists a bounded measurable
function $\tg_n:E_n\to\R$ such that for each
$\prfi{\sF_n(t)}$-stopping time $\tau$, we have
\begin{align}
  \label{smg}
  \een{G \circ \theta_n(\tau) \giv \sF_n(\tau) } =
  \tg_n(X_n(\tau)),\text{
    $\bP_n$-a.s.~ on } \set{\tau<\infty},
\end{align}
where $\een{\cdot}$ denotes the expectation operator with respect to
$\bP_n$. In fact, we only need the Markov property to hold on
deterministic times and at the following stopping times
\begin{align}
  \label{tgt}
  \tgt_n := \inf \sets{ s\geq t}{ X_n(s) = x_n}.
\end{align}

Given $n\in\N$, let $A_n$ be a nondecreasing additive
functional on $D(E_n)$, i.e., an
$\prfi{\sF_n(t)}$-adapted, \cd~and nondecreasing process with the
property that $A_n(0)=0$ and
\begin{align}
  \label{additive}
  A_n(t+s) = A_n(t)  +(A_n(s)) \circ \theta_t \eforall t\geq 0,
  \bP_n\text{-a.s.}
\end{align}
for each $s\geq 0$.

We recall that for
each  L\' evy subordinator (a nondecreasing L\' evy process) $X$ there
exists a nonnegative function $\Phi$---called the Laplace exponent of
$X$---such that $\ee{\exp{-\mu X_t}}=\exp{ -t \Phi(\mu)}$. We refer
the reader to \cite[Chapter 12]{Whi02} for the definition and the
important properties of Skorokhod's $M_1$-topology.

\begin{theorem}
\label{main-1} Suppose that the following conditions hold:
\begin{enumerate}[label=(\arabic*)]
  \item \label{ttoz} For all $t\geq 0$ and $\eps>0$, we have $\ppn{ \tgt_n
      \geq t+\eps} \to 0$ as $n\to\infty$.
  \item \label{equi} There exist a pair of continuous functions
  $a,b:[0,\infty) \to [0,\infty)$ such that $a(0)=b(0)=0$, $b$ is concave
  and unbounded,  and
  \begin{align*}  \een{ \big. b\big(A_n(t) - A_n(s)\big)  } \leq a(t-s)
  \end{align*}
  for all $0\leq s<t < \infty$ and $n\in\N$.
  \item \label{reso} There exists a constant $\ld >  0$ such 
    that the limit
  \begin{align}
    \label{lim_res}
    R^{\ld,\mu} = \lim_n \een{ \int_0^{\infty}
      \exp{\Big. -\ld t - \mu
        A_n(t)}\, dt}
  \end{align}
  exists for all $\mu \geq 0$.
\end{enumerate}
Then the sequence $\prfi{A_n(t)}$ converges in law, under the
Skorokhod's $M_1$-topology, to a L\' evy subordinator whose Laplace
exponent $\Phi(\mu)$ is given by
\begin{align*} \Phi(\mu) = \frac{1}{R^{\ld,\mu}} - \ld\, .\end{align*}
\end{theorem}
Since it is quite technical, but not central to the main focus of the
paper, we relegate the proof of Theorem \ref{main-1} to Appendix
\ref{app:proof}.

\subsection{Convergence in a diffusion framework}
We start by outlining the diffusion framework in which
Theorem~\ref{main-2} holds. Throughout the paper, we use the standard
diffusion terminology without further explanation; for a succinct but
comprehensive summary of the basic notions and standard properties of
one-dimensional diffusions we refer the reader to
\cite[Chapter~II]{BorSal02}. For a  complete treatment, see the
canonical text \cite{ItoMcK74}.

We consider a sequence of one-dimensional diffusion laws without
explosion or killing, whose state spaces $I_n$ are convex subsets of
$\R$; we set
$l_n = \inf I_n \in [-\infty,\infty)$ and $r_n = \sup I_n \in
(-\infty,\infty]$.
Let $(\bP_n^x)_{x \in I_n}$ denote the associated Markov family of
probability measures on the canonical space $C([0,\infty); I_n)$,
where $\bP_n^x$ denotes the law of the process started at $x$ at
time $0$. As usual, the mixture laws
$\bP_n^{\nu} := \int \bP_n^x \, \nu(dx)$ correspond to 
nondeterministic initial conditions.

The probability laws $(\bP_n^x)_{x \in I_n}$ determine the
characteristics of the diffusions: the speed
measures $\seq{m_n}$ and the strictly increasing and continuous
scale functions $\seq{s_n}$. We note that the speed measure is
assumed to be defined on the endpoints $l_n$ and/or $r_n$
whenever they are included in $I_n$. We impose the following standing
assumption:
\begin{assumption}
  $m_n(I_n) = 1$ for all $n \in \N$.
  \label{mno}
\end{assumption}
\begin{remark}
\label{along}
\ 
\begin{enumerate}
  \item 
Since the speed measure $m_n$ is defined only up to a multiplicative
constant, the Assumption \ref{mno} above can be 
weakened to $m_n(I_n) < \infty$ without loss of generality. 
The benefit of this normalization $m_n(I_n) = 1$ is that
$m_n$ becomes the unique invariant probability measure for 
$(\bP_n^x)_{x \in I_n}$.
\item Assumption \ref{mno} is equivalent (see \cite[par.~12, p.~20]{BorSal02})
  to the requirement of 
positive recurrence, namely,
\begin{align*}
  \eenup{x}{T_n^y} < \infty \eforall\, x,y \in I_n, \ewhere 
T_n^y = \inf\sets{ t > 0}{ X_n(t) = y }.
\end{align*}
This implies, in particular, that
a boundary point $b_n \in \set{l_n, r_n}$ is 
nonsingular if $b_n \in I_n$ and natural
otherwise.
\end{enumerate}
\end{remark}

\medskip

For $n \in \N$, let $A_n$ be a continuous and nondecreasing
additive functional of $X_n$.
More precisely, $A_n$ is a continuous process,
defined on the space $C([0,\infty); I_n)$ with values in $[0,\infty)$, with the property that
for each $s\geq 0$ and $x \in I_n$,  we have \[ A_n(t+s) = A_n(t)
  + (A_n(s)) \circ \theta_t,
  \eforall t \geq 0,\ \bP_n^x\text{-a.s.}\] 
In the statement and proof of Theorem \ref{main-2} below we use
$A_n$ to ``kill'' the process $X_n$. This helps us 
analyze the behavior of $A_n$ and $X_n$ together by studying a single, killed diffusion. 
With that in mind, we let $X^{\mu}_n$ be the process with the
same dynamics as $X_n$, but killed at the ``rate'' $\mu\, dA_n(t) $, with
$\mu>0$. More precisely, 
let $\tau_n$ be an exponentially
distributed random variable with rate $1$, 
defined on suitable extension of the underlying probability space, 
which is 
independent of
$\prfi{X_n(t)}$ under each $\bP^x_n$, $x\in I_n$. We define for each
$t \geq 0$
\begin{align}
  \label{killed}
  X^{\mu}_n(t) =
  \begin{cases}
    X_n(t), & t< T_n^{\Delta} \\
    \Delta, & t\geq T_n^{\Delta},
  \end{cases}
  \ewhere\
  T^{\Delta}_n = \inf\sets{u\geq 0}{ \mu
    A_n(u)  \geq \tau_n},
\end{align}
with $\Delta$ denoting an isolated ``cemetery'' state added to the state
space $I_n$.  This can be equivalently described in terms of the killing
measure $k_n^{\mu}$ of $X_n^{\mu}$, given by $k_n^{\mu} =
  \mu K_n$ where $K_n$ is the \emph{representing (Revuz)
  measure of $A_n$}. More precisely, $K_n$ is the measure on $I_n$ with 
  the property (see
\cite[par.~23., p.~28]{BorSal02}) that
\begin{align*}
  A_n(t) = \int_0^t L_n(t,y)\, K_n(dy) \eforall t\geq 0,\ \bP^x_n\text{-a.s.}
  \eforall x \in I_n,
\end{align*}
where $L_n(t,y)$ denotes the (diffusion) local time of $X_n$ at the
level $y$, accumulated up to time $t$.
In the particular case when $A_n(t) = \int_0^t g_n(X_n(u))\,
  du$, the definition of the local time as an occupation density with
respect to the speed measure $m_n$ implies that $K_n(dy) = g_n(y)\, m_n(dy)$
(see \cite[par.~23., p.~28]{BorSal02}).

\medskip

For $x_n \in I_n$, $\mu \geq 0$ and $\ld>0$ let 
$\zetam_n:I_n \to (0,1]$ be given by
\begin{align}
  \label{zetam}
  \zetam_n(x) = \eenx{ e^{-\ld T^{x_n}_n}_n \inds{ T^{x_n}_n <
  T^{\Delta}_n}}, x \in I_n.
\end{align}
Since $\ppnx{ \tm_n > T^{x_n}_n \giv \sF^X_{T^{x_n}_n}} = \exp{-\mu
  A_n(T^{x_n}_n)}$, $\bP^x_n$-a.s.~by construction,
we also have 
\begin{align}
  \label{zetamA}
  \zetam_n(x) = \eenx{ e^{-\ld T^{x_n}_n - \mu A_n(T^{x_n}_n)}}, x \in I_n.
\end{align}
We observe that the restrictions of $\zetam_n$ to $ I_n \cap (-\infty,
x_n]$  (resp.~$I_n \cap [x_n, \infty)$) coincide (see
\cite[par.~11., p.~10]{BorSal02})
with the \emph{decreasing fundamental solution} $\vpm_n$  
(resp.~\emph{increasing fundamental
solution} $\psim_n$) associated to $X^{\mu}$:
\begin{align}
  \label{fund-vp}
  \begin{split}
    \vpm_n(x) =
    \begin{cases}
      \eenx{ e^{-\ld T^{x_n}_n} \inds{T^{x_n}_n <\tm_n
        }}, & x\geq x_n,\\
      1\big/\eenup{x_n}{ e^{-\ld T^{x}_n} \inds{T^x_n <
          \tm_n }}, & x < x_n,
    \end{cases} \\
    \psim_n(x) =
    \begin{cases}
      \eenx{ e^{-\ld T^{x_n}_n} \inds{T^{x_n}_n < \tm_n }},
      & x\leq x_n,\\
      1\big/\eeup{x_n}{ e^{-\ld T^{x}_n} \inds{T^x_n < \tm_n}},
      & x > x_n.
    \end{cases}
  \end{split}
\end{align}
Note that we normalize the fundamental solutions 
so that $\vpm_n(x_n) = \psim_n(x_n)=1$, and observe that the case $\mu=0$ corresponds to the
fundamental solutions $\vpz_n$ and $\psiz_n$ of the original 
(not killed) processes. Moreover, we have
\begin{align}
  \label{zeta-w}
\zetam_n = \vpm_n \wedge \psim_n.
\end{align}

\medskip

The main result of this section,  Theorem \ref{main-2} below, provides
sufficient and readily verifiable conditions for convergence to a
subordinator (see also Remark \ref{after} following it for additional
intuition behind and clarification of these conditions). The following notation will be used in the statement of
the theorem: for two probability measures $\mu$ and  $\nu$, and a
constant $x$, we write $ \nu \preceq_x \mu$ if 
\begin{align}
  \label{fmu}
  \int f\, d\nu \leq \int f\, d\mu,
\end{align}
for each nonnegative function $f$ which is nonincreasing on
$(-\infty,x]$ and nondecreasing on $[x,\infty)$. 
It is not difficult to see that this is equivalent to each of the
following two conditions: \ i) 
There exist two random variables $X$ and $Y$,
with distributions $\mu$ and $\nu$,
respectively, such that  $Y$ is between $x$ and $X$, i.e., 
  $Y \in [X \wedge x, X \vee x], \ \text{a.s.}$, and 
\ ii) The cumulative distribution functions $F_{\mu}$ and $F_{\nu}$
  satisfy:
  \begin{align*}
    F_{\mu}(y) \leq F_{\nu}(y) \efor y \geq x \eand 
    F_{\mu}(y) \geq F_{\nu}(y) \efor y \leq  x.
  \end{align*}
For a probability measure $\nu_n$ on $I_n$, $\bP_n^{\nu_n} \circ X(t)^{-1}$
denotes the marginal distribution of the coordinate map $X(t)$ under
$\bP_n^{\nu_n}$. 

  \begin{theorem}
  \label{main-2} For $n \in \N$ pick $x_n \in I_n$ and a probability
  measure $\nu_n$ 
  on $I_n$, and let  the function $\zetam_n$ be given by 
  \eqref{zetam}.
Suppose that Assumption \ref{mno} as well as the 
following conditions hold:
\begin{enumerate}[label=(\alph*)]
  \item \label{conv-m} 
   $\bP_n^{\nu_n} \circ X_n(t)^{-1} \preceq_{x_n} m_n$  for all $t\geq
    0$, $n \in \N$ and $\lim_n \int \zeta^0_n(x)\, m_n(dx) = 1$ 
  \item \label{supk} There exist a pair of continuous functions
  $a,b:[0,\infty) \to [0,\infty)$ such that $a(0)=b(0)=0$, $b$ is concave
  and unbounded,  and
  \begin{align*}
    \eennn{ \big. b\big(A_n(t) - A_n(s)\big)  } \leq a(t-s)
  \end{align*}
  for all $0\leq s<t < \infty$ and $n\in\N$.
  \item \label{conv} 
   We have $\lim_n \int \zetam_n(x)\, \nu_n(dx) = 1$ for all 
 $\mu\geq 0$ and 
      the limit $\Phi(\mu) := \lim_n
      \Phi_n(\mu)$ exists in $\R$, where
    \begin{align}
      \label{Psin}
      \Phi_n(\mu):= \mu \frac{\int \zetam_n(x)\, K_n(dx)}{\int
        \zetam_n(x)\, m_n(dx)},
    \end{align}
    and $K_n$ is the representing measure of $A_n$.
\end{enumerate}
Then the $\bP^{\nu_n}_n$-laws of the additive functionals $A_n$
converge weakly, with respect to Skorokhod's $M_1$ topology,
to the L\' evy subordinator with the Laplace exponent $\Phi$.
\end{theorem}

\begin{proof} 
As Feller ($C_b \to C_b$) processes, 
diffusions fit into the framework and satisfy the preconditions of 
the abstract convergence result stated before Theorem \ref{main-1}. The
proof proceeds by checking the three conditions of Theorem \ref{main-1} in order.

\medskip

\paragraph{\emph{Condition} \ref{ttoz} of Theorem \ref{main-1}.} 
For a fixed $\eps>0$, let $f(x) =  \ppnx{ T^{x_n} \geq \eps}$. 
The function $f$ is nondecreasing to the right of $x_n$ and
nonincreasing to its left. The condition 
$x_n \preceq \bP_n^{\nu_n} \circ X(t)^{-1} \preceq m_n$
implies that
\begin{align*}
  \eennu{ f( X_n(t)) } = \int f\, d \prn{\bP_n^{\nu_n} \circ X(t)^{-1}} \leq \int
  f\, d m_n = \eenmn{ f(X_n(t))}. 
\end{align*}
The Markov property and invariance of $m_n$ yield
\begin{align*}
  \ppnnu{\tgt \geq t+\eps}
  &= \eeupdown{\nu_n}{n}{\ppupdown{X_n(t)}{n}{ T^{x_n} \geq \eps}}
  = \eeupdown{\nu_n}{n}{f( X_n(t))}
  \\ & \leq \eeupdown{m_n}{n}{f( X_n(t))} 
   = \eeupdown{m_n}{n}{f(X_n(0))} 
   = \ppupdown{m_n}{n}{ T^{x_n} \geq \eps}.
\end{align*}
Finally, using Markov's inequality and the second part of condition \ref{conv-m}, 
we obtain
\begin{align*}
   \ppupdown{m_n}{n}{ T_n^{x_n} \geq \eps} 
  &\leq \oo{1-e^{-\ld \eps}}  \eenmn{1- e^{-\ld T_n^0}}
   = \oo{1-e^{-\ld \eps}} \int (1-\zeta^0_n(x)) \, m_n(dx) \to 0.
\end{align*}

\medskip

\paragraph{\emph{Condition} \ref{equi} of Theorem \ref{main-1}} 
Condition \ref{supk} in Theorem \ref{main-2} is identical to it.

\medskip

\paragraph{\emph{Condition} \ref{reso} of Theorem \ref{main-1}} 
For $x\in I_n$ we set
\begin{align*}
  R_n(x) = \eenx{\int_0^{\infty} e^{-\ld t - \mu A_n(t) }\,
    dt} \eand 
    Q_n(x) = \eenx{\int_0^{T^{x_n}} e^{- \ld t - \mu A_n(t) } \, dt}.
\end{align*}
The strong Markov property implies that
\begin{align*}
  R_n(x) &= \eenx{
    \int_0^{T^{x_n}} e^{- \ld t - \mu A_n(t) }\, dt +
    \int_{T^{x_n}}^{\infty} e^{- \ld t - \mu A_n(t)}  \, dt}\\
  &= Q_n(x) +
  \eenx{ e^{- \ld T^{x_n} - \mu
        A_n(T^{x_n}) } \int_{T^{x_n}}^{\infty}
    e^{-\ld (t-T^{x_n}) - \mu \prn{ A_n(t) - A_n(T^{x_n}) }}\, dt
  } 
 \\ &= Q_n(x) + \eenx{ e^{-\ld T^{x_n} - \mu A_n(T^{x_n}) }} R_n(x_n). 
 \\ &= Q_n(x) + \zetam_n(x) R_n(x_n).
\end{align*}
We have $0\leq R_n(x) \leq 1/\ld$ as well as
\begin{align*}
  0 &\leq  Q_n(x) \leq  \eenx{  \int_0^{T^{x_n}} e^{-\ld t}}=
  \oo{\ld} \prn{1-\eenx{ e^{-\ld T^{x_n}}}} = 
  \oo{\ld} (1 - \zeta^0_n(x)) \leq \oo{\ld} (1-\zetam_n(x)).
\end{align*}
Therefore, the first part of condition \ref{conv-m} implies that
\begin{align*}
  \abs{\eennn{ \int_0^{\infty} e^{-\ld t - \mu A_n(t)}\, dt} -
  R_n(x_n)} \movealign{10} \leq \int_{I_n} \Bprn{ Q_n(x) + R_n(x_n)(1
  - \zetam_n(x))}\, \nu_n(dx) \\
  & \leq \frac{2}{\ld} \int_{I_n} \bprn{1- \zetam_n(x)}\, \nu_n(dx)
  \to 0.
\end{align*}

To analyze the behavior of $R_n(x_n)$ as $n \to \infty$, we
observe that $R_n(x_n) = U_n f(x_n)$ for $f \equiv 1$, where $U_n$
is the resolvent operator associated to $X_n^{\mu}$, i.e., 
\begin{align*}
  U_n f(x) &=
  \eenx{ \int_0^{\infty} e^{-\ld t} f(X_n^{\mu}(t))\, dt} =
  \eenx{ \int_0^{T^{\Delta}_n} e^{-\ld t} f(X(t))\, dt}
           \\ &= \eenx{ \int_0^{\infty} f(X(t)) e^{-\ld t -\mu
           A_n(t)}\, dt}, x \in I_n.
\end{align*}
We assume, first, that $x_n \in \Int I_n$. 
According to \cite[Theorem (50.7), p.~293]{RogWil00v2} and the
subsequent remark, the
kernel of the resolvent operator $U_n$  
is absolutely continuous with respect to $m_n$ and we have the
following expression for $R_n({x_n}) = U_n 1 (x_n)$:
\begin{align*}
  R_n({x_n}) = (\wm_n)^{-1} \int_{I_n} \zetam_n(y) \, m_n(dy),
\end{align*}
where $\wm_n$ is the Wronskian. Its value 
is given by
\begin{align}
 \label{Wronsk}
   \wm_n 
   &= 
   \vpm_n(x) D^- \psim_n(x) - \psim_n(x)  D^- \vpm_n(x) 
\end{align}
where
the right-hand side does not depend on the choice of $x \in \Int I_n$
and 
\begin{align*}
D^{\pm} f(x) &= \lim_{\eps \downto 0} \frac{ f(x\pm\eps) -
  f(x)}{  s_n(x\pm \eps) - s_n(x)}. 
\end{align*}
The functions $\vpm_n$ and $\psim_n$ 
are generalized solutions to the Poisson equation
$\ld u - \sG_n u = 0$ where $\sG_n$ is the infinitesimal generator of
$X^{\mu}_n$
(see \cite[Section II.1,
  par.~10.,pp.~18-19]{BorSal02}). In the case $x_n \in \Int I_n$, 
  this implies that 
  for  $a<b$ with $a,b\in \Int I_n$ with $x_n \in (a,b)$,  we have
  \begin{align*} 
    \ld \int_{[x_n,b]} \vpm_n(x)\, m_n(dx) + \mu \int_{[x_n,b]}
    \vpm_n(x)\,  K_n(dx) 
    &= D^+ \vpm_n(b) - D^- \vpm_n(x_n), \eand \\
    \ld \int_{[a,x_n)} \psim_n(x)\, m_n(dx) + \mu \int_{[a,x_n)}
    \psim_n(x)\,  K_n(dx)
    &= D^- \psim_n(x_n) - D^- \psim_n(a).
  \end{align*}
  It follows from \eqref{Wronsk} that, for $x_n \in \Int I_n$, we
  have $w_n^{\mu} = D^-\psim_n(x_n) - D^-\vpm_n(x_n)$, so that, when added
  together, the previous two equalities yield
  \begin{align}
    \label{refan}
    \ld \int_{[a,b]} \zetam_n(x)\, m_n(dx) + \mu
    \int_{[a,b]} \zetam_n(x)\,  K_n(dx) 
= D^+ \zetam_n(b) - D^- \zetam_n(a) + \wm_n.
  \end{align}
  According  to  \cite[Section II.1,
  par.~10.,pp.~19]{BorSal02} for $l_n = \inf I_n$ we have
  \begin{align}
    \label{bdd}
  D^+ \psim_n(l_n) = \lim_{a \downto l_n} D^- \psim_n(a) = 
  \begin{cases}
  \ld \psim_n(l_n)
  m(\set{l_n}) + \mu \psim_n(l_n) K(\set{l_n}), & l_n \in I_n,\\
  0, & l_n \not\in I_n.
  \end{cases}
  \end{align}
  The analogous statement holds for the right
  end-point $r_n$ of $I_n$. So, by letting $a\downto l_n$ and $b
  \upto r_n$ in \eqref{refan}, we get the following expression
  \begin{align}
    \label{wm-int}
    \wm_n = \ld \int_{I_n} \zetam_n(x)\, m_n(dx) + \mu
    \int_{I_n} \zetam_n(x)\,  K_n(dx).
  \end{align}
  When $x_n$ lies on the boundary of $\Int I_n$ (say $x_n =
  l_n \in I_n$) \eqref{wm-int} still holds without modification. 
  The proof follows the same pattern: we first obtain the expression 
  $\wm_n= D^+\psim_n(l_n) - D^+\vpm_n(l_n)$ by letting
  $x \downto l_n$ in \eqref{Wronsk} and then use the weak-solution
  property and the boundary behavior \eqref{bdd} at $l_n$.

  Now that \eqref{wm-int} is established, we use it to derive the
  following identity:
  \begin{align*}
    \frac{1}{R_n(x_n)}  - \ld  = 
    \frac{\wm_n}{\int_{I_n} \zetam_n(x)\, m_n(dx)} - \ld= 
    \mu \frac{\int_{I_n} \zetam_n(x)\, K_n(dx)}
    {\int_{I_n} \zetam_n(x)\, m_n(dx)}\,.
  \end{align*}
It, along with the second part of our condition \ref{conv}, implies the condition \ref{reso} of Theorem \ref{main-1}.
\end{proof}

\begin{remark} \label{after} The existence of the limit of the functions $\Phi_n$ in
  condition \ref{conv} is, in a sense, the only ``hard'' requirement of
  Theorem \ref{main-2}. Let us comment on the role of the other
  conditions and circumstances in which they hold:
  \begin{enumerate}
    \item    \label{it1} The condition $\bP_n^{\nu_n} \circ X(t)^{-1} \preceq_{x_n}
    m_n$ ensures that the ``reset point'' $x_n$ and the initial
    distribution $\nu_n$ are chosen so
    that the diffusion approaches the stationarity
    ``from within'', relative to $x_n$.  
    It holds automatically when
    $\nu_n = m_n$. Another common situation when it holds (as is
    easily proved via a coupling argument) is
    when $x_n$ is one of the endpoints of $I_n$ and $\nu_n \preceq_{x_n} m_n$. 
    A special case of that, when $\nu_n = \delta_{0}$ and  $I_n =
    [0,r_n)$ is used in the Wright-Fisher example below. 
  \item \label{it2}
    The limiting condition $\int \zeta^0_n(x)\, m_n(dx) \to 1$
    of \ref{conv-m} is
    equivalent to $\ppupdown{m_n}{n}{ T^{x_n}_n > \eps} \to 0$ and
    ensures ``faster and faster mixing'' for the sequence $\seq{X_n}$. It
    is automatically satisfied, for example,  when 
    $X_n(t) = c_n X( v_n t )$ where $X(t)$ is a stationary diffusion
    and $v_n \to \infty$ with appropriate conditions placed on $c_n$.  
    It holds in numerous other cases, 
    such as the one considered in the following section. 
  \item \label{it3}
    The first part of \ref{conv}, namely, $\int \zetam_n(x)\, \nu(dx)
    \to 1$, guarantees that the accumulation of the additive
    functional $A_n$ by the  first ``reset time'' $T^{x_n}_n$ time can
    be asymptotically ignored. A limiting theorem could be proved
    under a much weaker version of this condition, but the
    limiting process would exhibit a nontrivial independent initial
    jump, drawn from a possibly different jump distribution, before
    shifting into the subordinator dynamics. 
  \item \label{it4}
    Condition \ref{supk} is straightforward to check if the functional
    $A_n$ is of the form $A_n(t) = \int_0^t g(X_n(u))\, du$, which
    will be of interest in the following section. A simple
    sufficient condition in that case is that the expectation
    $\eennu{g(X_n(t))}$ be bounded, uniformly in $n\in\N$ and $t$ on
    compacts. This will clearly be the case when $\nu_n = m_n$ and $g$
    is uniformly integrable over all $m_n$. More flexibility, allowed
    by a general function $b$, is needed when $K$ is not absolutely
    continuous with respect to the Lebesgue measure, for example, when
    $A_n$ is a local time. 
\end{enumerate}

\end{remark}

\section{Wright-Fisher Diffusions} \label{sec:Wright-Fisher} In this
section, we present an application of Theorem \ref{main-2} to a
sequence of Wright-Fisher diffusions, the study of which is the
practical motivation for this work. We also consider a sequence of
Feller (CIR) diffusions, which can be seen as limiting cases of
rescaled Wright-Fisher processes when the right end-point of the state
space converges to $+\infty$. 

\subsection{The scaling regime.} Using the notation of Section
\ref{sec:diffusions}, we consider
a sequence $\seq{X_n}$ of diffusions on the state space $I_n=[0,1)$,
parameterized by three sequences $\seq{\tau_n}$,
$\seq{\al_n}$ and $\seq{\beta_n}$ of strictly positive numbers. Their
infinitesimal generators are given by
\begin{align}
  \label{gen-WF}
  \sG_n f(x) = \oo{ \tau_n} x(1 -x) f''(x) +
  \oo{\tau_n} (\al_n (1 - x)  - \be_n x)  f'(x),
\end{align}
for $f \in C_c^{2}((0,1))$. We assume that $X_n(0)=0$, i.e.,
that the initial distribution $\nu_n$ is $\delta_0$.

As our focus will be on the regime $\al_n \to 0$, the Feller
condition at the left boundary point will not be satisfied, rendering 
the boundary at $0$ nonsingular. We impose instantaneous reflection
there, i.e., set $m_n(\set{0})=0$,
since it is not only the cleanest  
choice mathematically, but also best suited to our intended application
to neuroscience. Moreover, 
we assume throughout that $\beta_n > 1$. This guarantees that 
$X_n$ is well-defined in the sense that it does not leave the state space
$[0,1)$ in finite time. Indeed, when $\beta_n > 1$ 
the right boundary $r_n = 1$ is not an exit boundary 
(see \cite[eq.~(6.19), p.~240]{KarTay81}). 

All of this leads to the 
following expressions for the derivatives of the scale
functions and the densities of the speed measures
\begin{align*}
  s_n'(x) & = \tau_n B(\al_n, \be_n)\, x^{-\al_n} (1-x)^{-\be_n}, \\
  m'_n(x) & = \oo{B(\al_n, \be_n)} \, x^{\al_n-1} (1-x)^{\be_n-1},
\end{align*}
where $B(\al,\be) =\Gamma(\al) \Gamma(\be)/\Gamma(\al+\be)$
is the Beta function and $\Gamma(\cdot)$ is the Gamma function. We
refer the reader to \cite[Example 8, p.~239]{KarTay81} for the
details, as well as for a discussion of various properties and
features of the Wright-Fisher diffusion.

\medskip

The scaling regime adopted in this section is
\begin{align}
  \label{pars}
  \tau_n \to 0,  \
  \beta_n = \beta > 1 \eand \frac{\al_n}{\tau_n} \to
  \gamma \text{ for some } \gamma \in (0,\infty),
\end{align}
with the sequence $\seq{A_n}$ of additive functionals given by
\begin{align}
  \label{add-WF}
  A_n(t) =  \oo{\tau_n} \int_0^t X_n(u)\, du.
\end{align}
The particular choices made in \eqref{pars} are partly dictated by
modeling considerations, and partly by their mathematical interest.
Moreover, this regime is essentially forced by the choice that
$\seq{\beta_n}$ be constant, the assumptions of Theorem \ref{main-2},
and the requirement that the limit be nondeterministic. Indeed, as
can be verified directly, we have
\begin{align}
  \label{exp}
  \eenmu{\oo{\tau_n} \int_0^1 X_n(t)\, dt} & =  \frac{
    \al_n}{\tau_n} \frac{ 1}{\al_n+\be}, \quad  \eand
  \\ \label{var}
  \Var^{m_n}_n\bkt{ \oo{\tau_n} \int_0^1 X_n(t)\, dt}
  &=
  \frac{ 2 \be}{ (\al_n+\be)^3 (1+\al_n+\be)}
  \frac{
    e^{-\frac{\al_n+\be}{\tau_n}} - 1 + \frac{ \al_n+\be}{ \tau_n} }{
    \frac{\al_n+\be}{\tau_n}}\, . 
\end{align}
Whence, it follows that $1/\tau_n$ is, indeed, the proper
scaling for $\int_0^1 X_n(t)\, dt$, and that, given that scaling,
the limiting variance is nontrivial only if the limit of
$\al_n/\tau_n$ exists in $(0,\infty)$.

\subsection{Decreasing fundamental solutions} Next, we turn to
the decreasing fundamental solutions $\seq{\vpm_n}$ of the killed
diffusion defined in \eqref{fund-vp} above. We take the analytic
approach and characterize $\vpm_n$, up to a multiplicative constant,
as the unique decreasing solution of the following second-order
ordinary differential equation: 
\begin{align}
  \label{ode-wf}
  \sG_n u(x)
  - \prn{ \ld + \frac{\mu}{\tau_n}  x } u(x)  =0,\ x\in (0,1),\  u(0)=1.
\end{align}
Since the right boundary is natural, no boundary
conditions need to be imposed at the right endpoint.

\medskip

In order to pass to a limit in the following subsection we 
require a more precise 
understanding of the structure of the solution of
\eqref{ode-wf} than is provided by the general theory. Given that we
are working with a polynomial diffusion, i.e., a diffusion with an
infinitesimal generator whose coefficients are polynomials, it is
plausible to expect that the solutions to \eqref{ode-wf} admit power-series
expansions amenable to further analysis. This direct approach also 
turns out to be the most convenient.
To see this, let
us consider a candidate solution $\um_n$ specified as
\begin{align}
  \label{um}
  \um_n(x) = \sum_{k=0}^{\infty} a_n(k) (1-x)^k,
\end{align}
where the coefficient sequence $\seqkz{a_n(k)}$ is defined by the
following recurrence relations:
\begin{align}
  \label{rec-ic}
  a_n(0) &=1,\  a_n(1) = \frac{ \ld\tau_n   + \mu
  }{\be}
  \eand \\
  \label{rec}
  a_n(k)  &= c_n(k-1) a_n(k-1) - c_n(k-2) a_n(k-2), \efor k\geq
  0,
\end{align}
where
\begin{align}
  \label{rec-coeff}
  c_n(k-1) &= \frac{ \ld\tau_n +\mu +   (k-1)(k+\al_n+\be-2)}{ k
    (\beta+k-1)}, \eand \\
  c_n(k-2) &=
  \frac{   \mu    }{ k(\beta+k-1)}.
\end{align}
These recursions are obtained by coefficient matching when
\eqref{um} is formally inserted in \eqref{ode-wf}. Moreover, 
although the equation \eqref{ode-wf} is of second order, the value of
the coefficient $a_n(1)$ is determined by the equation itself
due to the degeneracy of ellipticity at the right boundary. On the
other hand, the choice $a_n(0)=1$ is only a normalization.
\begin{lemma}
\label{lem:a-growth} For each $\eps>0$ there exist constants
$C_{\eps}>0$ and $N_{\eps}\in\N$ such that
\begin{align}
  \label{bd-a}
  \abs{a_n(k)} \leq C_{\eps} k^{-(2-\eps)} \eforall k\in\N
  \eand n\geq N_{\eps}.
\end{align}
\end{lemma}
\begin{proof}
Given $\eps \in (0,1)$, we set $K^1_{\eps} = 8 \be/\eps$ and
choose $N_{\eps} \in \N$, such that $\al_n < \eps/4$ and $\tau_n <
  1$ for $n\geq N_{\eps}$. For $k\geq K^1_{\eps}$ and $n\geq
  N_{\eps}$, we have $\frac{2-\al_n}{k+\be-1} \geq
  \frac{2-\eps/2}{k}$, so that
\begin{align*}
  0 \leq c_n(k-1) &=
  1 -  \frac{(2-\al_n)}{k+\beta-1} + \frac{2+\ld\tau_n-\al_n - \be +
    \mu}{k(k+\be-1)}  \leq 1 - \frac{ \eta + \eps/2}{k} + \frac{\rho}{
    k^2},
\end{align*}
where $\eta = 2 - \eps$, $\old = \sup_n \ld\tau_n<\infty$ and
$\rho = 2+\old+\mu$. We also have
\begin{align*}
  0 \leq  c_n(k-2) & \leq \frac{\mu}{k^2}.
\end{align*}
Let $b_n(k) = \abs{a_n(k)}/k^{-\eta}$, so that, for $k\geq
  K^1_{\eps}$ and $n\geq N_{\eps}$ we have
\begin{align*}
  b_n(k) & =
  \prn{\Big. 1-\frac{\eta+\eps/2}{k} + \frac{\rho}{k^2} }
  \frac{b_n(k-1) (k-1)^{-\eta} }{k^{-\eta}} +
  \frac{\mu}{k^2} \frac{b_n(k-2) (k-2)^{-\eta} }{k^{\eta}}\\
  &\leq \max\prn{\Big. b_n(k-1), b_n(k-2) } f\prn{ 1/k },
\end{align*}
where
\begin{align*}
  f(x)  = (1-x)^{-\eta} \prn{\Big. \rho x^2 - x
    (\eta+\eps/2) + 1} + \mu
  (1-2x)^{-\eta} x^2 \efor x<1/2.
\end{align*}
Clearly, $f$ is $C^1$ on $[0,1/2)$, $f(0) = 1$ and $f'(0) =
  -\eps/2$, so there exists $x_0>0$ such that $f(x)\leq 1$ for $x\in
  [0,x_0]$, i.e.,
\begin{align} b_n(k) \leq \max( b_n(k-1), b_n(k-2)) \efor k\geq
  \label{b-bnd}
  K_{\eps}:=\max(K^1_{\eps}, 1/x_0).\end{align}
The absolute values of the coefficients $c_n(k)$ and the initial
conditions $a_n(0), a_n(1)$ admit $n$-independent bounds, which
implies that
\begin{align}
  \label{B-max}
  B(k):=\sup_n b_n(k) \leq \sup_n k^{\eta}\abs{a_n(k)}<\infty
  \eforeach k \in \N.
\end{align}
Combined with \eqref{b-bnd}, the finiteness of $B(k)$ in
\eqref{B-max} above implies that, for $n \geq N_{\eps}$ we have
\begin{equation*}
  \abs{a_n(k)} k^{-\eta} \leq C_{\eps}:=\max_{k\leq K_{\eps}}
  B(k)<\infty. \qedhere
\end{equation*}
\end{proof}

\begin{proposition}
The function $\um_n$ is well-defined by \eqref{um} on $[0,2]$,
real analytic on $(0,2)$, and we have $\vpm_n = \um_n$ on
$[0,1]$, up to a multiplicative constant.
\end{proposition}
\begin{proof}
The bounds in \eqref{bd-a}, for $\eps < 1$, immediately imply that
the series \eqref{um} converges absolutely on $[0,2]$ and 
defines a continuous function there. Analyticity on $(0,2)$ then
follows from the fact that the radius of convergence is at least
$1$. In particular, we can differentiate them term by term and then
perform an easy calculation using \eqref{rec-ic} and \eqref{rec} to
conclude that $\um_n$ solves \eqref{ode-wf} on $(0,1)$ and that
$\um_n(1)=1$, $(\um_n)'(1) = -(\ld\tau_n +\mu)/\be$.

Next, we show that $\um_n$ is strictly decreasing. Arguing by
contradiction, assume first that $\umnp(x) \geq 0$ for some
$x\in (0,1)$, and let $x_0 \in (0,1]$ be the supremum of all such
$x$. Strict negativity of the derivative $(\um_n)'(1)$ implies that
$\umnp<0$ in a neighborhood of $1$, and so, $x_0 <1$. Hence,
$\umnp(x_0)=0$ and $\umnp(x)<0$ for $x\in (x_0,1)$, which, in turn,
implies that $\umnpp(x_0) \leq 0$. Since $\um_n$ satisfies
\eqref{ode-wf}, we must have $\um_n(x_0)\leq 0$. This is in
contradiction with the fact that $\um_n(1)=1 $ and $\umnp(x)\leq 0$
for all $x\in [x_0,1)$.

Finally, we appeal to the general theory of one-dimensional
diffusions (see \cite[Section II.1, par.~10., pp.~18-19]{BorSal02}),
which states that $\vpm_n$ is the unique, up to a multiplicative
constant, decreasing solution to \eqref{ode-wf} (no boundary
conditions needed). Therefore, $\vpm_n$ and $\um_n$ agree on $(0,1)$,
up to a multiplicative constant. By continuity, the same is holds on
$[0,1]$.
\end{proof}

\subsection{Limiting behavior of fundamental solutions.} Next, we analyze the limiting behavior of the sequence $\seqn{\vpm_n}
  = \seq{\um_n}$. Since the coefficients in \eqref{rec} and the
initial conditions \eqref{rec-ic} converge to finite values as
$n\to\infty$, the solutions converge, and we set $a(k):=\lim_n
  a_n(k)$. Moreover, the limiting coefficients satisfy the following
(limiting) recursive equation
\begin{align}
  \label{lim-rec-ic}
  a(0) &=1,\  a(1) = \frac{\mu}{\beta}, \quad \eand \\
  \label{lim-rec}
  a(k)  &=
  \frac{ \mu  +   (k-1)(k+\be-2)}{ k (k+\be-1)}
  a(k-1) -  \frac{   \mu   }{ k(\beta+k-1)} a(k-2)
  \efor k\geq 2.
\end{align}
It is easily checked that \eqref{lim-rec-ic} and \eqref{lim-rec}
admit an explicit solution, namely,
\begin{align}
  \label{expl-a}
  a(k) = \frac{ \mu^k}{ k! (\beta)_k},
\end{align}
where $(\beta)_k := \beta (\beta + 1) \dots (\beta+k-1)$ is the
Pochhammer symbol (also known as the rising factorial). Therefore, we set
\begin{align*}
  \vpm(x) := \sum_{k=0}^{\infty} \frac{\mu^k}{k! (\beta)_k}
  (1-x)^k,
\end{align*}
with absolute convergence for all $x$, and note that

\begin{align*}
  \vpm(x) =
  \Gamma(\be) (\mu(1-x))^{-(1+\be)/2} I_{\be-1}(2 \sqrt{\mu(1-x)})
\end{align*}
where $I_{\nu}$ is the modified Bessel function of the first kind of order $\nu$.

\medskip

By Lemma \ref{lem:a-growth} applied with $\eps= 1/2$, we have
$\abs{a_n(k) - a(k)} \leq C k^{-3/2}$ for some $C>0$, all $k\in\N$,
and large enough $n\in\N$. Therefore, we can use the dominated
convergence theorem to conclude that $\lim_n \sum_k \abs{a_n(k) -
a(k)}=0$, so that
\begin{align}
  \label{unif-um}
  \sup_{x\in [0,1]} \abs{\vpm_n(x) - \vpm(x)}
  \leq \sum_{k=0}^{\infty} \abs{ a_n(k) -
    a(k)} \to 0 \text{ as } n\to\infty.
\end{align}
Since $m_n \to \delta_{0}$ weakly, where $\delta_0$ denotes the
Dirac measure concentrated at $0$, the uniform convergence of
\eqref{unif-um} above implies that
\begin{align}
  \label{limm}
  \int \vpm_n(x)\, m_n(dx) \to \vpm(0) = \sum_{k=0}^{\infty}
  \frac{\mu^k} {k! (\be)_k}= \Gamma(\be) \mu^{-(1+\be)/2}
  I_{\be-1}(2\sqrt{\mu}).
\end{align}

To compute the limit $\int \vpm_n(x)\, K_n(dx)$, we first note that
the density $K_n'(x)$ of $K_n$ with respect to the Lebesgue measure
(see \cite[Section II.1, par. 9, p.~17]{BorSal02})
satisfies
\begin{align}
  \label{isbeta}
  \frac{\tau_n}{\al_n} (\al_n + \be) K_n'(x)
  & =
  \frac{\Gamma(\al_n+1 + \be)}{\Gamma(\al_n+1) \Gamma(\be)}
  x^{(\al_n+1)-1} (1-x)^{\be -1},
\end{align}
where the right-hand side above can be recognized as the probability
density of the beta distribution with parameters $\al_n+1$ and
$\be$. As $n\to \infty$, these distributions converge weakly to 
the beta distribution with parameters $1$ and $\beta$. Thus, 
by \eqref{unif-um}, we have
\begin{align}
  \notag
  \int \vpm_n(x)\, K_n(dx)
  &\to   \frac\gm\be  \int \vpm(x)\,
  \be (1-x)^{\beta-1}\, dx
  =  \gm \sum_{k=0}^{\infty} \frac{\mu^k}{k! (\be)_k}
  \int_0^1 (1-x)^{\be -1 + k}\, dx\\
  \label{limk}
  &= \frac\gm\be  \sum_{k=0}^{\infty}
  \frac{\mu^k}{k! (\be+1)_k } \, dx
  = \gamma \Gamma(\be) \mu^{-\be/2} I_{\be}(2 \sqrt{\mu})\,.
\end{align}

\subsection{The main result.}
We now present the main result of this section.
\begin{theorem}
\label{main-3} Consider the sequence $\seq{X_n}$ of
Wright-Fisher diffusions on $[0,1)$ with generators given by
\eqref{gen-WF}, started at $X_n(0)=0$, instantaneously reflected at
$0$, and under the scaling regime \eqref{pars}. 
The sequence $\seq{ A_n }$ of rescaled and integrated diffusions, given by
\begin{align*}
  A_n(t) = \oo{\tau_n} \int_0^t X_n(u)\, du, \ t\geq 0,
\end{align*}
converges weakly, under Skorokhod's $M_1$-topology, to a
L\' evy subordinator  whose Laplace exponent is given by
\begin{align}
  \label{Phi-wf}
  \Phi(\mu) & = \gamma \sqrt{\mu} \frac{
    I_{\be}(2\sqrt{\mu})}{I_{\be-1}(2 \sqrt{\mu})},
\end{align}
where $I_{\nu}$ is the modified Bessel function of the first kind
with index $\nu$.
\end{theorem}
\begin{proof}
  We verify the conditions (a), (b) and (c) 
  of Theorem \ref{main-2}:
\begin{enumerate}[label=(\alph*)]
  \item 
  Since $\nu_n = \delta_0$ and $x_n=0$, we have 
  $\nu_n  \preceq_{x_n} m_n$;  according 
to item (\ref{it1}) of Remark \ref{after}, this implies that
$\bP^{\nu_n}_n \circ X_n(t)^{-1} \preceq_{x_n} m_n$ 
for all $t \geq 0$ and $n \in \N$. 
Thanks to \eqref{zeta-w}, the function $\zeta^m_n$ of \eqref{zetam} 
coincides with $\vpm_n$, so the second part of condition
\ref{conv-m}  follows from  
\eqref{limm} with $\mu=0$. 
\item 
We use 
\eqref{exp} and the fact that
$\bP^{\nu_n}_n \circ X_n(t)^{-1} \preceq_{x_n} m_n$ in 
\begin{align*}
  \eennn{A_n(t) - A_n(s)} = \eennn{ \oo{\tau_n} 
  \int_s^t X_n(u)\, du} \leq 
  \eenmn{\oo{\tau_n} \int_s^t X_n(u)\, du} \leq C(t-s)
\end{align*}
for some constant $C$.
\item
Since $\nu_n = \delta_0$, the first part of condition (c) is trivially
satisfied.  For the second one it suffices to 
take the quotient of \eqref{limm} and
\eqref{limk}. \qedhere
\end{enumerate}
\end{proof}

\subsection{Properties of the limiting subordinator} 
We continue this
section with some facts about the limiting Laplace functional $\Phi$
and the limiting subordinator which we denote by $A$. Given that it is only a scaling parameter, we
assume throughout that $\gamma=1$ for simplicity.

\begin{enumerate}[wide, labelindent=0pt, itemsep=1em]
  \item
  The function $\Phi$ appeared as a conditional Laplace exponent
  in the literature (see \cite[eq.~(9.s7), p.~348]{PitYor81}) in the
  following context.
  Let $X$ denote the Bessel process of index $\nu=\beta-1$ (i.e., dimension
  $\delta = 2 \beta$) started at $X_0=1$.  We define the
  process $A$ by
  \begin{align}
    A(t) = 2 \int_0^{\tau^1(t)} \inds{X_u\leq 1}\, du,\ t \in
    [0,L^1_{\infty}),
  \end{align}
  where $L^1$ and $\tau^1$ are the local and inverse local
  times of $X$ at level $1$.
  Since $X$ is transient for $\nu>0$, the process
  $\tau$ eventually jumps to $+\infty$, a.s.,
  making $A$ a proper killed subordinator.
  It turns out, however, that for each $t>0$,
  conditionally on its lifetime
  exceeding $t$ (i.e., on $\set{L_{\infty}^1 \geq t}$),
  $A$ is a L\' evy subordinator on $[0,t]$ with
  Laplace exponent $\Phi$.
  We refer the reader to
  \cite[Remark 9.8 (ii), p.~349]{PitYor81} for the
  outline of the idea of the proof, or to \cite[Corollary 2,
    p.~6]{PitYor03} for a more comprehensive treatment.
  \item Since $\Phi$ is a Laplace exponent of an infinitely-divisible
  distribution supported by $[0,\infty)$, it
  admits a L\' evy-Khinchine representation of the form
  \begin{align}
    \label{LK}
    \Phi(\mu) = b \mu + \int_0^{\infty} (1-e^{-\mu x}) \Pi(dx) \efor
    \mu>0,
  \end{align}
  where $b\geq 0$ and $\Pi$ is a measure on $(0,\infty)$ such that
  $\int \min(1,x)\, \Pi(dx)<\infty$. By \cdl{10.30}{4}, we have
  $\lim_{x\to\infty} \sqrt{2 \pi x} e^{-x} I_{\nu}(x) = 1$, so
  \begin{align*}
    \lim_{\mu \to \infty} \oo{\mu} \Phi(\mu) &= 2
    \lim_{x\to\infty}  \oo{x} \frac{ I_{\be-1}(x)}{ I_{\be}(x)} = 0 ,
  \end{align*}
  which implies that $b=0$, i.e., that $A$ has no drift.
  \item According to \cite[Theorem 1.9, p.~886]{IsmKel79}, the
  function \[ \Psi(\mu) = \frac{2\be}{ \sqrt{\mu}} \frac{
      I_{\be}(\sqrt{\mu}) }{ I_{\be-1}(\sqrt{\mu})} \ \ewith \ \mu>0,\]
  is a Laplace transform of the infinitely divisible distribution
  with density
  \begin{align*}
    f(y) = 4 \be \sum_n \exp{ - \jbn^2 y }, y\geq 0 ,
  \end{align*}
  where $\seq{\jnn}$ is an enumeration of the set of strictly positive
  zeros of the Bessel function $J_{\nu}$ of index $\nu$. We have
  \begin{align*}
    \Psi \prn{ \mu } = \frac{4\be }{\mu} \Phi\prn{ \frac{\mu}{4} },
  \end{align*}
  so that
  \begin{align*}
    \int_0^{\infty} e^{-\mu y} f(y)\, dy
    &=
    4\be \int_0^{\infty} \frac{1-e^{-\frac{\mu}{4}  x}}{\mu}
    \Pi(dx) \\
    &= \be \int_0^{\infty}
    \int_0^x e^{-\frac{\mu}{4}  y}\, dy\, \Pi(dx) \\
    &= \int_0^{\infty} e^{-\mu z}\, 4 \be\, \Pi
    \prn{ \Big[\oo{4} z,\infty\Big)}\, dz,
  \end{align*}
  for all $\mu>0$. We conclude that the L\' evy measure $\Pi$ is
  absolutely continuous with respect to the Lebesgue measure, with
  density
  \begin{align}
    \label{pi-dens}
    \pi(x) = \oo{\be} f'(4x) =
    \sum_n (2\jbn)^2 \Exp{- (2\jbn)^2  x}, \quad x>0.
  \end{align}
  \item Thanks to \eqref{pi-dens} above, we have
  \begin{align}
    \label{Ray}
    \int x^{r}\,  \Pi(dx) =  \sum_n  \int_0^{\infty} x^r (2\jbn)^2
    \Exp{ - (2\jbn)^2 x}\, dx
      = 4^{-r} \Gamma(1+r) \sum_n \jbn^{-2r}
  \end{align}
  Since the zeros of the Bessel functions grow approximately linearly;
  more precisely (see \cdl{10.21}{19}),
  \begin{align*}
    \jbn \sim \pi\prn{ n+\tot( \be-3/2) } + O(1/n), 
  \end{align*}
  for each $T\in [0,\infty)$ we have
  \begin{align*}
    \ee{ \sum_{t \leq T} (\Delta A_t)^r} =
    \begin{cases}
      +\infty, & r\leq 1/2, \eand  \\
      <+\infty, & r>1/2.
    \end{cases}
  \end{align*}

  \item When the L\' evy exponent $\Phi$ is analytic in a neighborhood
  of
  $0$, as in our case, the sequence $\seq{\kappa_n}$ of
  cumulants is defined using the Maclaurin expansion
  \begin{align*}
    \Phi(\mu) = \sum_{n=0}^{\infty} (-1)^n \kappa_n \frac{\mu^n}{n!},
  \end{align*}
  of the function $\Phi$. Their importance stems from the fact that
  they are the moments of the jump measure, i.e.,
  \begin{align*}
    \kappa_n = \int_0^{\infty} x^n \Pi(dx), \efor n\in\N.
  \end{align*}
  The explicit expression \eqref{Ray} show that, in our case, we
  have
  \begin{align*}
    \kappa_n =  4^{-n} n! \, \sigma_n(\beta-1) \ewhere
    \sigma_n(\nu) = \sum_m (\jnn)^{-2m}.
  \end{align*}
  The function $\sigma_n$ is known as the \emph{Rayleigh function},
  and satisfies the following simple convolution identity (see
  \cite[Eq.~(20), p.~531]{Kis63}), useful for efficient computation of
  cumulants and moments:
  \begin{align*}
    \sigma_n(\nu) = \oo{\nu+n} \sum_{k=1}^{n-1}  \sigma_k(\nu)
    \sigma_{n-k}(\nu), \quad \sigma_1(\nu) = \oo{4(\nu+1)}.
  \end{align*}
  Once the cumulants are known, the moments $m_n = \ee{ A(t)^n}$,
  $n\in\N$, of the distribution of $A(t)$ can be efficiently computed
  by using the following well-known recursive relationship, which is,
  in turn, a direct consequence of the formula of Fa\` a-di-Bruno:
  \begin{align*}
    m_{n+1} = t \sum_{i=0}^n (-1)^i \binom{n}{i} \kappa_{i+1}
    m_{n-i}, \ m_0 = 1.
  \end{align*}
  In particular, as is the case for any L\' evy process, $m_{n}$ is a
  polynomial in $t$ of order at most $n$.
  \item We have the following simple continued-fraction expansion of
  the Laplace exponent $\Phi$ (see \cite[Theorem 6.3,
    p.~206]{JonThr81}):
  \begin{align*}
    \Phi(\mu) = \cfrac{\mu}{\be+\cfrac{\mu}{(\be+1)+
        \cfrac{\mu}{(\be+2) + 
          \cdots \vphantom{\cfrac{1}{1}}
          }}}
  \end{align*}

\end{enumerate}

\subsection{The Feller (CIR) diffusion}
We conclude this section with an example in the context
of a sequence of Feller (CIR) diffusions. It can be seen as an
extension of the results on Wright-Fisher diffusions since
Feller processes can be interpreted as 
limits of properly rescaled Wright-Fisher diffusions as the right
endpoint of the state space tends to $+\infty$. 

We take $I_n = [0,\infty)$, $\nu_n=\delta_0$, $m_n(dx) = m'_n(x)\, dx$
and $s_n(x) = \int_0^x s'_n(y)\, dy$, where
\begin{align}
  \label{mnsn}
  m'_n(x) & = \frac{\beta^{\al_n}}{\Gamma(\al_n) } x^{\al_n-1}
  e^{-\be x} \eand s'_n(x) = e^{\beta x} x^{-\alpha_n},\  x\in [0,\infty),
\end{align}
and consider the limiting behavior of $A_n(t) = n \int_0^t X_n(u)\,
du$ in the regime 
 \begin{align}
   \label{reg-F}
 \beta > 0, \alpha_n < 1 \eand 
n \alpha_n \to \gamma > 0. 
 \end{align}
The associated infinitesimal generator 
$\sG_n$ is given by
\begin{align}
  \label{gen-F}
  \sG_n f = n x f''(x) + n (\alpha_n - \beta x) f'(x) \efor f \in
  C^2_c((0,\infty)).
\end{align}
We note that the Feller condition will
not be satisfied at $0$ in our parameter regime, so the requirement that
$m_n(\set{0})=0$, which is implicit in \eqref{mnsn}, makes the left boundary instantaneously reflective.

\smallskip

Let $U(a,b,\cdot)$ denote Kummer's $U$-function (see \cdl{13.2}{6}), so
that $u(x) = U(a,b,x)$ solves Kummer's differential equation
\begin{align}
  \label{Kummer}
  x u''(x) +(b-x) u'(x)  - a\,u(x) = 0 \efor x\in (0,\infty).
\end{align}
A direct computation shows that for $\mu>0$, the function
\begin{align*}
  \vpm_n(x) = \oo{\Gamma(V_n)} e^{L x} U(V_n, \al, S x)
\end{align*}
where
\begin{align*}
  L = \frac{\beta - \sqrt{\be^2+4\mu}}{2},\
  R = \frac{\beta + \sqrt{\be^2+4\mu}}{2},\
  S= R - L,\
  V_n = \frac{ \ld/n - \al_n L}{S}
\end{align*}
satisfies
\begin{align}
  \label{CIR-G}
  \sG_n \vpm_n(x) - (\ld + \mu n x) = 0.
\end{align}
For $a>0$ and $x>0$, the integral representation (see
\cdl{13.4}{4}))
\begin{align*}
  U(a,b,x) = \oo{\Gamma(a)} \int_0^{\infty} e^{-t x} t^{a-1}
  (1+t)^{b-a-1}\, dt \efor a,x\in (0,\infty),
\end{align*}
can be used to justify the identity
\begin{align}
  \label{int-rep}
  \vpm_n(x) = \oo{\Gamma(V_n)} \int_0^{\infty} \Exp{ -x (St - L) }
  t^{V_n-1} (1+t)^{\al -V_n -1}\, dt.
\end{align}
Since $S >0$ and $L<0$, we conclude immediately that $\vpm_n$ is
positive and strictly decreasing. This is enough (see Section
II.1, par.~10., pp.~18-19 of \cite{BorSal02}) to identify $\vpm_n$ out of all
solutions of \eqref{CIR-G} as the decreasing fundamental solution,
up to a multiplicative constant.

The representation \eqref{int-rep} yields
\begin{align*}
  \frac{\Gamma(V_n)}{\Gamma(\al_n)} \int \vpm_n(x)\, m_n(dx)
  &=
  \int t^{V_n-1} (1+t)^{\al_n-V_n-1}
  \oo{\Gamma(\al_n)} \int x^{\al_n-1} e^{-(R + St) x}
  dx\, dt \\
  &= \int (S t+ R)^{-\al_n} t^{V_n-1}
  (1+t)^{\al_n-V_n-1}\, dt  \\
  &= \int_0^1 r^{-1+V_n} (R (1-r) + Sr)^{-\al_n} \, dr ,
\end{align*}
where we use the substitution $r \leftarrow t/(1+t)$ to get
the last equality. Similarly
\begin{align*}
  \frac{\Gamma(V_n)}{\Gamma(\al_n+1)} \int \vpm_n(x) x\, m_n(dx)
  &= \int_0^1 r^{-1+V_n} (1-r) (R (1-r) + Sr)^{-\al_n-1}  \, dr
\end{align*}
Combining the integral representations given above allows us to write
\begin{align*}
  \Phi_n(\mu)
  =
  \frac{\int \vpm(x) \mu \, n x m_n(dx)}{ \int \vpm(x)\, m_n(dx)}
  =
  \frac{n \alpha_n \mu (1+V_n)\int_0^1 \frac{r^{-1+V_n}
  (1-r)}{B(V_n,2)} (R (1-r) + Sr)^{-\al_n-1}  \, dr}{ \int_0^1
  \frac{r^{-1+V_n}}{B(V_n,1)} (R (1-r) + Sr)^{-\al_n} \, dr} \, .
\end{align*}
Since $V_n \to 0$,
the sequence of beta distributions with parameters $(V_n,B)$ 
converges weakly to the Dirac mass at $0$
for any $B>0$. Moreover, since
$R, S>0$, we have
\begin{align*}
  (R(1-r) + S r)^{-\al_n} \to 1 \eand (R(1-r) + S r)^{-\al_n-1}
  \to (R(1-r) + S r)^{-1}
\end{align*}
uniformly on $[0,1]$. Therefore, since $n\alpha_n \to \gamma$, we
obtain
\begin{align*}
  \Phi(\mu) =   \lim_n \frac{\int \vpm(x) \mu \, n x m_n(dx)}{ \int
    \vpm(x)\, m_n(dx)}  =
  \frac{ 2 \gamma \mu}{ \beta+ \sqrt{\beta^2+4\mu}} =
  \frac{\gamma \beta}{2}   \prn{ \sqrt{1 + \frac{4  \mu}{\beta^2}} - 1 },
\end{align*}
the central condition of Theorem \ref{main-2}. The remaining conditions of
Theorem \ref{main-2} are verified as in the proof of Theorem
\ref{main-3}; we remark that the bound 
$\eenmn{A_n(t)-A_n(s)} \leq C(t-s)$ follows from the fact that
the barycenters of $\seq{m_n}$ scale as $1/n$ as $n\to\infty$.

The discussion above leads to the following result:
\begin{theorem}
\label{main-4}
\label{main-wf} Consider the sequence $\seq{X_n}$ of
Feller diffusions on $[0,\infty)$ with generators given by
\eqref{gen-F}, started at $X_n(0)=0$, instantaneously reflected at
$0$, and under the scaling regime \eqref{reg-F}. 
The sequence $\seq{ A_n }$ of rescaled and integrated 
diffusions, given by
\begin{align*}
  A_n(t) = n \int_0^t X_n(u)\, du, \ t\geq 0,
\end{align*}
converges weakly, under Skorokhod's $M_1$-topology, to a
L\' evy subordinator  whose Laplace exponent is given by
\begin{align}
  \label{Phi-f}
\Phi(\mu) =  
\frac{\gamma \beta}{2}   \prn{ \sqrt{1+ \frac{4  \mu}{\beta^2}} - 1 }.
\end{align}
\end{theorem}
\begin{remark}
\label{peony}
We recognize \eqref{Phi-f} as the
Laplace exponent of the inverse-Gaussian distribution with the mean
$\gamma/\be$ (typically denoted by $\mu$) and the scale parameter
$\gm^2/2$ (typically denoted by $\ld$).
\end{remark}

\appendix
\section{Proof of Theorem \ref{main-1}}
\label{app:proof}
For the sake of clarity, we divide the proof into four steps. The
stopping time $\tgt_n$, defined in \eqref{tgt}, will appear 
numerous times, so we introduce the following shortcut:
\begin{align*}
  \tau_n =  \tgt_n,
\end{align*}
where the dependence on $t$ will always be clear from context.
\smallskip

\emph{Step 1.} For $n\in\N$, let $\bQ_n$ denote the law of $A_n$ on
$D([0,\infty))$. Our first claim is that condition \ref{equi}
implies that the family $\seq{\bQ_n}$ is tight under the $M_1-$topology on $D([0,\infty))$. It will be enough to prove this fact
for the restrictions of our processes to all bounded intervals of the
form $[0,T]$ with $T>0$  (see \cite[section 12.9, pp.~414-416]{Whi02}).
We base our approach on \cite[Theorem 12.12.3, p.~426]{Whi02} which
gives two necessary and sufficient conditions
for tightness under $M_1$ on $D([0,T])$. We remark that the modulus of
continuity $w_s$ (see \cite[eq.~(4.4), p.~402]{Whi02}), 
which is a major component of the second condition in the
general case, vanishes for monotone processes.
With this simplification, the two conditions for tightness
become:
\begin{enumerate}
  \item[(i)] For each $\eps>0$ there exists $c>0$ such that
  \begin{align*}
    \ppn{ A_n(T) > c } < \eps \eforall n\in\N.
  \end{align*}
  \item[(ii)] For each $\eps>0$ and $\eta>0$, there exists $\delta>0$ such that
  \begin{align*}
    \ppn{   A_n(\dl)  \geq \eta} < \eps \eand
    \ppn{   A_n(T)  - A_n(T-\dl) \geq \eta} < \eps.
  \end{align*}
\end{enumerate}
Condition \ref{equi} implies, via Markov's inequality,
that for any $0 \leq s < t$ we have
\begin{align}
  \label{Markov}
  \ppn{ A_n(t) - A_n(s)  \geq x} =
  \ppn{ b(A_n(t) - A_n(s))  \geq b(x)} \leq \frac{a(t-s)}{b(x)}.
\end{align}
To obtain (i), we use the fact that $b(x) \to
  \infty$ as $x \to \infty$ and choose $x$ such that $a(T)/b(x) <
  \eps$. For (ii), we first take $x$ small enough to ensure $b(x) \leq \eta$, and then choose $\dl>0$ so that
$a(\dl)/b(x) < \eps$.

The $M_1$-topology is metrizable so, by Prohorov's theorem, there exists an $M_1$-weakly convergent subsequence
\begin{align}
  \label{subseq_choice}
  \seqk{\bQ_{n_k}} \text{ of } \seq{\bQ_n},
\end{align}
and we denote its limit by $\bQ$.  To keep the notation
manageable in what follows, we do not relabel the convergent
subsequence $\seqk{\bQ_{n_k}}$ and proceed as if the
original sequence $\seq{\bQ_n}$ converged. 
To prepare for the next steps, let
$(\Omega,\sF,\bP)$ be a probability space on which a
nondecreasing \cd~process $A$ with law $\bQ$ is defined.

\medskip

\emph{Step 2.} We begin by transforming condition \ref{ttoz} into a
more useful form. It implies that, for each
$t\geq 0$, there exists a strictly increasing sequence $\seqkz{n_k}$
in $\bN_0$ such that $n_0=0$ and for each $k\in\N$,
\begin{align*}
  \ppn{ \tau_n > t + (k+1)^{-1}} < (k+1)^{-1}
  \eforall n  > n_k.
\end{align*}
We then define the sequence $\seq{\eps_n}$ (which may depend on $t$)
by
\begin{align}
  \label{eps_n}
  \eps_{n} = k^{-1} \efor n_{k-1} < n \leq  n_{k}, \ k\in\N,
\end{align}
so that $\eps_n \to 0$ as $n\to\infty$. On the other hand, the
inequality
\begin{align*}
  \ppn{ \tau_n > t + \eps_n} =
  \ppn{ \tau_n > t + k^{-1}} < k^{-1} = \eps_n
  \efor n_{k-1} < n \leq n_k,
\end{align*}
implies that $\ppn{ \tau_n > t + \eps_n} < \eps_n$ for all $n$.
Consequently, condition \ref{ttoz} implies
the existence of a sequence $\seq{\eps_n}$ satisfying 
$\epsilon_n \to 0$
such that
\begin{align}
  \label{hit_faster}
  \ppn{ \tau_n > t + \eps_n} \xrightarrow[]{n \to \infty} 0 .
\end{align}

\emph{Step 3.} By \cite[Theorem 2.5.1(iv), p.~404]{Whi02}, there
exists a dense subset $\sT$ of $[0,\infty)$, that includes $0$,
such that $A_n \to A$ in the sense of finite-dimensional distributions on
$\sT$, i.e., such that for all $K\in\N$ and all $t_1,\dots, t_K\in
  \sT$ we have
\begin{align}
  \label{fdd}
  \prn{\Big. A_n(t_1),\dots, A_n(t_K) }
  \tol{\sD} \prn{\Big. A(t_1), \dots,
    A(t_K) }.
\end{align}
We fix $t,\delta \geq 0$ and define the sequence $\seq{F_n}$ of
random variables by
\begin{align*}
  F_n = f(A_n(t_1),\dots, A_n(t_K)),\efor K\in\N \eand
  0\leq t_1\leq \dots \leq
  t_K \leq t,
\end{align*}
where $t_1,\dots, t_K$, $t$, $t+\delta$ belong $\sT$ and $f:\R^K \to \R$
is continuous, bounded, and bounded away from $0$.
For each $n$ and
each bounded Lipschitz function $g:\R\to\R$, we have
\begin{align*}
  \een{F_n g\prn{\Big. A_n(t+\delta) - A_n(t)}}=
  I_n^1 + I_n^2 + I_n^3,
\end{align*}
where
\begin{align*}
  I_n^1 &= \een{F_n\,  g\prn{\big. A_n(t+\delta) - A_n(t) }
    \inds{\tau_n > t+\eps_n}}, \\
  I_n^2 & = \een{\Big.F_n \,
    \prn{\Big. g\prn{\big. A_n(t+\delta) - A_n(t)  } -
      g\prn{\big. A_n(\tau_n + \delta) - A_n(\tau_n)}}
    \inds{\tau_n \leq t+\eps_n} }\\
  \intertext{and}
  I_n^3 &=
  \een{\Big.F_n \,
    \prn{\Big.  g\prn{\big. A_n(\tau_n+\delta) - A_n(\tau_n) }
    }
    \inds{\tau_n \leq t+\eps_n }
  },
\end{align*}
with $\seq{\eps_n}$ given by \eqref{eps_n}.

\medskip

Let $C$ denote a generic constant, independent of $n$, but possibly
depending on $f$ and $g$. As is customary, we allow $C$ to change
from occurrence to occurrence. The relation \eqref{hit_faster} above
implies that
\begin{align}
  \label{e_one}
  \abs{I_n^1}
  \leq C\, \ppn{ \tau_n > t+\eps_n} \to 0
  \text{ as } n\to\infty.
\end{align}
Moving on to $I_n^2$, we pick a constant $c>0$ (to be determined
later) and split it further into two parts
\begin{align*}
  I^{2;\leq c}_n = I_n^2 \inds{A_n(t+\eps_n+\delta)\leq c} \eand
  I^{2;> c}_n = I_n^2 \inds{A_n(t+\eps_n+\delta)> c},
\end{align*}
which we estimate
separately.

Owing to the uniform boundedness of $F_n$ and the Lipschitz property
of $g$ we  have
\begin{align*}
  I_n^{2; \leq c} & =\een{
    \abs{\Big.
      g\prn{ \Big.
        A_n(t+\delta) - A_n(t)
      } -
      g\prn{ \Big.
        A_n(\tau_n+\delta) - A_n(\tau_n)}}
    \inds{\tau_n \leq t+\eps_n, A_n(t+\eps_n+\delta) \leq c} } \\
  &\leq \een{
    \abs{\Big.
      A_n(t+\delta) - A_n(t)  -
      A_n(\tau_n+\delta) + A_n(\tau_n)}
    \inds{\tau_n \leq t+\eps_n, A_n(t+\eps_n+\delta) \leq c} } \\
  &\leq
  C\een{\prn{\Big.
      \abs{
        A_n(\tau_n) - A_n(t)} + \abs{
        A_n(\tau_n+\delta) - A_n(t+\delta)}}
    \inds{\tau_n \leq t+\eps_n, A_n(t+\eps_n+\delta) \leq c}
  } \\
  &\leq C\een{\Big.
    \prn{\Big.
      A_n(t+\eps_n) - A_n(t) +
      A_n(t+\delta+\eps_n) - A_n(t+\delta)}
    \inds{A_n(t+\eps_n+\delta) \leq c}
  }. 
\end{align*}
Since $b$ is concave and $b(0)=0$, we have
\begin{align*}
  x \leq \frac{c}{b(c)} b(x) \eforall 0 < x\leq c,
\end{align*}
so that on $\set{A_n(t+\eps_n+\dl)\leq c}$ we have
\begin{align*}
  A_n(t+\eps_n) - A_n(t) \leq \frac{c}{b(c)} b\Big(A_n(t+\eps_n) -
  A_n(t)\Big)
\end{align*}
as well as
\begin{align*}
  A_n(t+\dl+\eps_n) - A_n(t+\dl) \leq \frac{c}{b(c)}
  b\Big(A_n(t+\dl+\eps_n) -A_n(t+\dl)\Big).
\end{align*}
Therefore,
\begin{align*}
  I_n^{2; \leq c} \leq
  C \frac{c}{b(c) }
  \een{
    b\Big(A_n(t+\eps_n) - A_n(t)\Big) +
    b\Big(A_n(t+\delta+\eps_n) - A_n(t+\delta)\Big)
  } \leq C \frac{c}{b(c)} a(\eps_n).
\end{align*}
Moreover, by the boundedness of $g$ and the estimate
\eqref{Markov}, we obtain
\begin{align*}
  I^{2; >c}_n &=\een{ \abs{\Big.
      g\prn{ \Big. A_n(t+\delta) - A_n(t) } -
      g\prn{ \Big. A_n(\tau_n+\delta) - A_n(\tau_n)}}
    \inds{\tau_n \leq t+\eps_n, A_n(t+\eps_n+\delta) > c} } \\
  &\leq C \ppn{ A_n(t+\eps_n+\delta) > c} \leq C
  \frac{a(t+\sup_n \eps_n+\delta)}{b(c)}.
\end{align*}
By taking $c$ sufficiently large we can make $I^{2; >c}_n$ arbitrarily
small, uniformly in $n$. With that $c$ fixed, we have $I^{2; \leq
      c}_n \to 0$ as $n \to \infty$, so that $\abs{I^2_n} \to 0$ as
$n\to\infty$.

Lastly, by \eqref{additive} and the strong Markov property
\eqref{smg}, for each $n\in\N$, there exists a bounded and measurable
function $\tg_n:E_n \to \R$ such that
\begin{align*}
  I_n^3 & =
  \een{  F_n g\prn{\Big.  A_n(\tau_n + \delta) - A_n(\tau_n) }
    \inds{\tau_n \leq t+\eps_n }}\\
  & =
  \een{  F_n g\prn{\Big. A_n(\delta) \circ \theta_{\tau_n} }
    \inds{\tau_n \leq t+\eps_n }}\\
  &=
  \een{  F_n \een{ g\prn{\Big. A_n(\delta) \circ \theta_{\tau_n} }
      \inds{\tau_n \leq t+\eps_n } \giv \sF_n(\tau_n)}}\\
  &= \een{ F_n \inds{\tau_n \leq t+\eps_n }} \tg_n(x_n).
\end{align*}
Thanks to \eqref{fdd},
\begin{align*}
  &\ee{f(A(t_1), \dots, A(t_K)) g( A(t+\delta) - A(t))} =
  \lim_n \een{ F g\prn{ A_n(t+\delta) - A_n(t) }}  \\
  & \qquad = \lim_n (I^1_n+I^2_n+I^3_n) = \lim_n \een{ F_n \inds{\tau_n
      \leq t+\eps_n}}\tg_n(x_n).
\end{align*}
As in \eqref{e_one} above, we have 
$\een{ F_n \inds{\tau_n > t+\eps_n}} \to 0$ so that
\begin{align*}
  \lim_n \een{ F_n \inds{\tau_n \leq t+ \eps_n}}
  &= \ee{ f(A(t_1),\dots, A(t_K))}.
\end{align*}
Since $f$ is bounded away from $0$, we conclude that
\begin{align}
  \label{lim_Q_f}
  \frac{
    \ee{ f(A(t_1),\dots, A(t_K)) g(A(t+\delta) -
      A(t))}}{
    \ee{ f(A(t_1),\dots, A(t_K))} } = \lim_n \tg_n(x_n).
\end{align}
As the process $A$ is \cd, \eqref{lim_Q_f} holds for all $K\in\N$,
and all $0 \leq t_1\leq \dots \leq t_K \leq t < \infty$, $\delta
  \geq 0$ --- not only those in $\sT$. Also, since the right-hand side
depends neither on $f$ nor on $t$, the random variable $A(t+\delta)
  - A(t)$ is independent of $\sigma(A_s, s\leq t)$ and its
distribution does not depend on $t$. In other words, $A$ has
stationary and independent increments. Since $A_n(0)=0$ for each $n$
and $M_1$-convergence implies convergence in distribution at $0$, we
conclude that $A(0)=0$, as well. Being right-continuous and
nondecreasing, $A$ is, therefore, a L\' evy subordinator.

\medskip

\emph{Step 4.} To close the loop and complete the proof, we use
condition \ref{reso}. The space $D([0,\infty))$ is $J_1-$separable,
where $J_1$ refers to Skorokhod's $J_1-$topology (see \cite[Section
  3.3., p.~78]{Whi02}). Since the $M_1-$topology is weaker than the $J_1-$topology 
(see \cite[Theorem 12.3.2, p.~398]{Whi02}), and $D([0,\infty))$ is
separable under $J_1$ (see \cite[p.112]{Bil99}), we have that
$D([0,\infty))$ is $M_1-$separable as well. Therefore, we can use the
Skorokhod representation theorem (see \cite[Theorem 3.2.2,
  p.~78]{Whi02}) to couple the laws of $\seq{A_n}$ and $A$ on the same
probability space such that $A_n \to A$ in $M_1$ almost surely. Next, we
recall that, for right-continuous, nondecreasing functions,
convergence on a dense set to a right-continuous function implies
convergence at every continuity point of the limit (see, e.g.,
the proof of \cite[Theorem 6.20, p.~142]{Kal21}, for the standard
argument). From this, we conclude that for nondecreasing functions
$M_1-$convergence implies convergence almost everywhere with respect
to Lebesgue measure. This is enough to establish that for any nonnegative,
continuous, and bounded function $f : \R^2 \to \R$, integral
functionals of the form
\begin{align}
  \label{int-func}
  y \mapsto \int_0^t f(u, y_u)\, du,
\end{align}
are continuous in the $M_1-$topology when restricted to the set of
nondecreasing functions in $D([0,\infty))$. The dominated convergence
theorem yields
\begin{align*}
  \een{ \int_0^{\infty} e^{-\ld t} e^{-\mu A_n(t)}\, dt} \to
  \ee{ \int_0^{\infty} e^{-\ld t} e^{-\mu A(t)}\, dt},
\end{align*}
for all $\ld>0$ and $\mu \geq 0$. Combined with condition
\eqref{lim_res}, this implies that for some $\ld>0$ we have
\begin{align*} \ee{ \int_0^{\infty} e^{-\ld t} e^{-\mu A(t)}\, dt}
  = R^{\ld,\mu} \eforall \mu\geq 0.\end{align*}
Moreover, since $A$ is a L\' evy subordinator, we have
\begin{align*}
  R^{\ld,\mu}
  &= \int_0^{\infty} e^{-\ld t} \ee{ e^{ - \mu A(t)}}\, dt
  = \int_0^{\infty} e^{-\ld t}  e^{-t \Phi(\mu)}\, dt =
  \oo{\ld+\Phi(\mu)},
\end{align*}
where $\Phi$ is the Laplace exponent of $A$. Since $\Phi$ completely
characterizes the distribution of $A(1)$, and, thus, the law of the
entire L\' evy process $A$, we conclude that the limit is the same
for each choice of a convergent subsequence in
\eqref{subseq_choice}. Hence, the entire sequence
$\seq{A_n}$ converges in law to $A$ under the $M_1-$topology.

\section*{Acknowledgments}

  During the preparation of this work the first named author was
  supported by the National Science Foundation under Career Award
  DMS-2239679, and the second named author by the  National
  Science Foundation under Grant DMS-2307729. Any opinions, findings
  and conclusions or recommendations expressed in this material are
  those of the author(s) and do not necessarily reflect the views of
  the National Science Foundation (NSF).


\newcommand{\etalchar}[1]{$^{#1}$}
\providecommand{\bysame}{\leavevmode\hbox to3em{\hrulefill}\thinspace}
\providecommand{\MR}{\relax\ifhmode\unskip\space\fi MR }
\providecommand{\MRhref}[2]{%
  \href{http://www.ams.org/mathscinet-getitem?mr=#1}{#2}
}
\providecommand{\href}[2]{#2}

\end{document}